\documentclass[18pt]{article}

\usepackage[margin=1.2in]{geometry}

\usepackage{amsfonts,amsmath,latexsym,amssymb,verbatim,amsbsy,amsthm,multicol}

\usepackage[colorlinks=true, pdfstartview=FitV, linkcolor=blue,
            citecolor=blue, urlcolor=blue]{hyperref}
\usepackage[usenames]{xcolor}
\definecolor{Red}{rgb}{0.7,0,0.1}
\definecolor{Green}{rgb}{0,0.7,0}
\usepackage{accents}
\usepackage{comment}
\usepackage{float}
\usepackage{graphicx}
\usepackage{textcomp}
\usepackage{mathrsfs}
\usepackage{mathtools}
\usepackage{dsfont}
\usepackage{enumitem}
\usepackage[capitalize,nameinlink,noabbrev]{cleveref}
\usepackage{empheq}

\numberwithin{equation}{section}

\theoremstyle{plain}
\newtheorem{THEOREM}{Theorem}[section]
\newtheorem{theorem}[THEOREM]{Theorem}
\newtheorem{corollary}[THEOREM]{Corollary}
\newtheorem{lemma}[THEOREM]{Lemma}

\theoremstyle{definition}

\newtheorem{case}{Case}

\theoremstyle{remark}
\newtheorem{remark}[THEOREM]{Remark}

\def \v {V^{(1)}}
\def \vt {\widetilde{V}^{(1)}}

\def \bu {{\bf u}}

\def \cC {\mathcal{C}}

\newcommand{\N}{\ensuremath{\mathbb{N}}}   
\newcommand{\R}{\ensuremath{\mathbb{R}}}   

\newcommand{\der}[2]{#1 \cdot \nabla #2}

\DeclareMathOperator{\supp}{supp} %
\newcommand{\ind}{\mathds{1}}
\newcommand{\dep}{\delta_p}
\newcommand{\tU}{\widetilde{U}^{(1)}}
\newcommand{\Ay}{A_y}
\newcommand{\tp}{\tilde{p}}
\newcommand{\ta}{\tilde{a}}
\newcommand{\tb}{\tilde{b}}
\newcommand{\td}{\tilde{d}}
\newcommand{\tbu}{\tilde{\bf u}}

\begin{document}

\title{On the locally self-similar blowup for the generalized SQG equation}
\author{Anne Bronzi\thanks{acbronzi@unicamp.br}, Ricardo Guimar\~aes\thanks{r192301@dac.unicamp.br}, Cecilia Mondaini\thanks{cf823@drexel.edu}}

\maketitle

\begin{abstract}
We analyze finite-time blowup scenarios of locally self-similar type for the inviscid generalized surface quasi-geostrophic equation (gSQG) in $\R^2$. Under an $L^r$ growth assumption on the self-similar profile and its gradient, we identify appropriate ranges of the self-similar parameter where the profile is either identically zero, and hence blowup cannot occur, or its $L^p$ asymptotic behavior can be characterized, for suitable $r, p$. Our results extend the work by Xue \cite{Xue16} regarding the SQG equation, and also partially recover the results proved by Cannone and Xue \cite{CX15} concerning globally self-similar solutions of the gSQG equation. 	
\end{abstract}

{\noindent \small
   }

\section{Introduction}\label{sec:intro}

This paper concerns the study of possible self-similar finite-time blowup scenarios for the generalized surface quasi-geostrophic equation (gSQG) in $\R^2$, namely
\begin{equation}\label{DSQG}
\left\{\begin{array}{ll}
\theta_t + \der{\bu}{\theta}  =0,  &  x \in \R^2, \,\, t > 0,\\
\bu  = - \nabla^\perp (- \Delta)^{-1 + \frac{\beta}{2}} \theta,  & x \in \R^2, \,\, t > 0, 
\end{array}\right.
\end{equation}
where $\beta \in (0,2)$ is a fixed parameter, $\theta = \theta(x,t)$ is an unknown scalar function,
and $\bu = \bu(x,t)$ denotes a velocity field. The latter is given in terms of $\theta$ according to the second equation in \eqref{DSQG}, where $\nabla^\perp = (-\partial_2, \partial_1)$, and $(-\Delta)^{-s/2}$, $0 < s < 2$, is the Riesz potential. From the definition of the Riesz potential (see e.g. \cite[Section V.1]{Stein}), it follows that $\bu$ can also be written as 
\begin{align}\label{velocity_field}
	\bu(x,t) = C_\beta P.V. \int_{\R^2} K_\beta(x - y) \theta(y,t) dy,
\end{align}
where 
\begin{align}\label{kernel}
	K_\beta(x) = \frac{x^\perp}{|x|^{2+\beta}}, \quad x \in \R^2 \backslash \{0\},
\end{align}
and $C_\beta$ is a constant depending only on $\beta$.

For $\beta = 0$, equation \eqref{DSQG} reduces to the vorticity formulation of the 2D incompressible Euler equations, a model for the evolution of inviscid and incompressible fluid flows in $\R^2$. Whereas for $\beta = 1$, \eqref{DSQG} coincides with the surface quasi-geostrophic equation (SQG), which models the evolution of surface temperature or buoyancy in certain large-scale atmospheric or oceanic flows \cite{Johnson1978,Blumen1978,Pedlosky1987,held1995surface}. Besides its physical relevance, the SQG equation has also received considerable attention due to its strong analytical and physical similarities to the 3D incompressible Euler equations \cite{PA94}. 

While global regularity for the 2D incompressible Euler equations is well established (see e.g. \cite{MBbook,MPbook}), the analogous question for the SQG equation remains completely open. Namely, it is not currently known whether smooth solutions of the SQG equation remain smooth for all time or develop singularities in finite time. In light of these results (or lack thereof), the generalized SQG equation \eqref{DSQG} was introduced in \cite{cordoba2005evidence} to naturally investigate the global regularity issue for a model that suitably interpolates between the 2D incompressible Euler equations ($\beta =0$) and the SQG equation ($\beta = 1$). Indeed, note from \eqref{DSQG} that for $\beta \in (0,1)$ the velocity field $\bu$ is more regular than at the right endpoint $\beta = 1$. The case $\beta \in (1,2)$, on the other hand, corresponds to a more singular velocity field, and was first considered in \cite{CCC}.

Despite several important advances, the question of global regularity or finite-time singularity formation also remains open for the gSQG equation, for any $\beta \in (0,2)$. Among the available results, local existence and uniqueness for the Cauchy problem associated to \eqref{DSQG} in the range $\beta \in (1,2)$ was shown in \cite{CCC} for any initial data in $H^4$, and later improved in \cite{hu2015existence} to any initial data in $H^s$, with $s > 1 + \beta$. An analogous local well-posedness result in $H^s$, $s > 1 + \beta$, for the more regular case $\beta \in (0,1]$ was shown in detail in \cite{inci2018well,yu2021local}. Additionally, a global regularity criterion in the case $\beta \in (0,1]$ was obtained in \cite{chae2011inviscid} with respect to the norm of a given solution in $\beta$-H\"older spaces, which generalizes a previous regularity criterion established for the SQG in \cite{PA94}. Specifically, \cite{chae2011inviscid} shows that $[0,T)$ is a maximal interval of existence for a solution $\theta$ of \eqref{DSQG} within the class $\cC^\sigma(\R^2) \cap L^q(\R^2)$, with $\sigma >1$ and $q >1$, if 
\begin{align}\label{reg:crit}
	\lim_{t \to T} \int_0^t \| \theta(\cdot, s)\|_{\cC^\beta(\R^2)} ds = \infty,
\end{align}
where $\cC^\gamma(\R^2)$, $0 < \gamma \leq 1$, denotes the space of $\gamma$-H\"older continuous functions on $\R^2$.

In addition to these analytical results, several computational studies were developed to numerically investigate the possibility of finite-time singularity formation for the SQG and gSQG in specific scenarios. Starting with the SQG equation, \cite{PA94} indicated a possible finite-time singularity in the form of a hyperbolic closing saddle, a suggestion that was later contested in \cite{ohkitani1997inviscid,CNS1998,constantin2012new} via further numerical tests, and eventually theoretically ruled out in \cite{cordoba1998nonexistence,cordoba2002growth}. On the other hand, in \cite{scott2011scenario,scott2014numerical}, analyzing an alternative scenario proposed by \cite{PierrehumbertEtAl1994,HoyerSadourny1982}, the authors found numerical evidence of a singularity occurring as a self-similar cascade of filament instabilities. Regarding the generalized SQG equation \eqref{DSQG}, numerical simulations were performed in \cite{cordoba2005evidence,Mancho2015,scott2019scale} focusing on the evolution of patch-like initial data, i.e. given by the indicator function of a spatial domain with smooth boundary \cite{Rodrigo2004,Rodrigo2005,gancedo2008existence,CCC}. Their results point to substantial evidence in support of the development of a corner-type singularity in finite time, which is approached in a self-similar manner.

While a rigorous proof of the formation of such singularities is still not available, these numerical studies provide a strong motivation to further investigate solutions of the gSQG equation that develop a finite-time singularity of self-similar type. Such solutions are defined with respect to the invariance of \eqref{DSQG} under the following scaling transformation $x, t, \theta \mapsto  \lambda x, \lambda^{1 + \alpha} t,$ $\lambda^{1 + \alpha - \beta} \theta$, with $\lambda \in \R^+$, $\alpha \in \R$; i.e. if $\theta$ is a solution of \eqref{DSQG}, then $\theta_{\lambda}(x, t) = \lambda^{1 + \alpha -\beta} \theta (\lambda x,$ $\lambda^{1+\alpha} t)$ is also a solution. Specifically, we say that, for a fixed scaling parameter $\alpha > -1$, a solution $\theta$ of \eqref{DSQG} on $\R^2$ and on a time interval $(0,T)$ is \emph{(globally) self-similar} if $\theta(x,t) = \theta_\lambda(x,t)$ for all $(x,t) \in \R^2 \times (0,T)$ and for all $\lambda > 0$. This is equivalent to being able to write $\theta$ as
\begin{align}\label{eq:theta:profile}
	\theta(x,t) = \frac{1}{t^{\frac{1+\alpha - \beta}{1 + \alpha}}} \Theta \left( \frac{x}{t^{\frac{1}{1+\alpha}}}\right) \quad \mbox{for all } (x,t) \in \R^2 \times (0,T),
\end{align}
for some function $\Theta: \R^2 \to \R$, which is called an associated \emph{self-similar profile}.

Some results on nonexistence of nontrivial globally self-similar solutions for the SQG and gSQG equations were obtained in \cite{Chae07,Chae11} and \cite{CX15}, respectively, by imposing suitable assumptions on the profile $\Theta$ and showing that $\Theta \equiv 0$ as a consequence, thus excluding the possibility of finite-time singularity of this type. More precisely, \cite{Chae07} assumed $\Theta \in L^{p_1}(\R^2) \cap L^{p_2}(\R^2)$ with $p_1, p_2 \in [1,\infty]$ and $p_1 < p_2$, whereas \cite{Chae11} considered $\Theta \in C^1(\R^2)$ such that $\lim_{|x| \to \infty} |\Theta(x)| = 0$. Both \cite{Chae07} and \cite{Chae11} utilize a particle trajectory and back-to-labels map approach to establish that $\Theta \equiv 0$. In \cite{CX15}, the authors analyze the gSQG equation in the case $\beta \in [0,1]$ and obtain an analogous result as in \cite{Chae07} while relying on a different technique centered on a local $L^p$ inequality satisfied by the profile $\Theta$.

We also mention the recent work \cite{garcia2022self}, where construction of a class of non-radial globally self-similar solutions with infinite energy of the gSQG  in the case $\beta \in (0,1)$ was obtained via suitable perturbations of a stationary solution. See additionally \cite{CastroCordoba2010}, where the authors consider solutions of the SQG equation in $\R^2$ of the form $\theta(x_1,x_2,t) = x_2 f_{x_1}(x_1,t)$ and construct a self-similar solution for the one-dimensional equation satisfied by $f$ which yields an infinite-energy solution for the SQG.

In this manuscript, we consider the more general case of solutions $\theta$ of \eqref{DSQG} that satisfy an equality as in \eqref{eq:theta:profile} only locally in space, namely with $(x,t) \in B_\rho(0) \times (0,T)$, for some $\rho > 0$\footnote{Note that, in contrast to the identity $\theta(x,t) = \theta_\lambda(x,t)$ for all $(x,t) \in \R^2 \times (0,T)$ and $\lambda > 0$ satisfied by $\theta$ in the globally self-similar case, a local version of \eqref{eq:theta:profile} with $x \in B_\rho(0)$ implies instead that $\theta(x,t) = \theta_\lambda(x,t)$ for all $x \in B_{\min\{\rho, \rho/\lambda\}} (0)$, $t \in \left(0, \min\left\{T, \frac{T}{\lambda^{1 + \alpha}}\right\} \right)$, and $\lambda >0$.}. Here, $B_\rho(0)$ denotes the ball in $\R^2$ centered at $0$ and with radius $\rho$. In fact, since, for any $x_0 \in \R^2$, $\tilde{\theta}(x,t) = - \theta(x-x_0, T-t)$, $(x,t) \in \R^2 \times (0,T)$, is also a solution of \eqref{DSQG} due to its spatial translation and time reversal symmetries, we may consider more generally solutions $\theta$ of \eqref{DSQG} that satisfy
\begin{align}\label{ss:x0:T}
	\theta(x,t) = \frac{1}{(T - t)^{\frac{1+\alpha - \beta}{1 + \alpha}}} \Theta \left( \frac{x - x_0}{(T - t)^{\frac{1}{1+\alpha}}}\right) \quad \mbox{for all } (x,t) \in B_\rho(x_0) \times (0,T),
\end{align}
for some $\rho > 0$ and some profile function $\Theta: \R^2 \to \R$. We refer to such $\theta$ as a \emph{locally self-similar solution}.

It is not difficult to see that if $\Theta \in \cC^\beta (\R^2)$ then condition \eqref{ss:x0:T} is indeed consistent with the regularity criterion from \cite{chae2011inviscid} when $\beta \in (0,1]$. Namely, \eqref{ss:x0:T} implies \eqref{reg:crit}, and hence $T$ represents a finite blowup time for $\theta$ in the class  $\cC^\sigma(\R^2) \cap L^q(\R^2)$, with $\sigma >1$ and $q >1$.

Analytical results regarding locally self-similar singularity scenarios were previously obtained for the $N$-dimensional incompressible Euler equations with $N \geq 3$ in \cite{ChaeShvydkoy2013,AR15,Xue2015}, 
for the SQG equation in \cite{Xue16}, and later for the 2D inviscid Boussinesq equations in \cite{Ju2020}. 
In particular, the result obtained in \cite{Xue16} yields, similarly as in the earlier works \cite{AR15,Xue2015}, suitable conditions on the self-similar profile under which existence of nontrivial $\Theta$ is only possible within an explicitly identified range of $\alpha$, i.e. $\Theta$ must necessarily be zero for $\alpha$ outside of this range. Moreover, any nontrivial profile corresponding to a value of $\alpha$ in this range must satisfy a certain asymptotic characterization of its $L^p$ average over sufficiently large regions in the spatial domain, for some $p >1$. As a consequence, this allows one to automatically exclude the existence of locally self-similar solutions with 
decaying profiles, while also guaranteeing the aforementioned asymptotic characterization of the $L^p$ average of certain non-decaying types of $\Theta$. 

Here, we obtain an extension of the result from \cite{Xue16} to the generalized SQG equation for all $\beta \in (0,2)$. Our main results are split between the cases $\beta \in (0,1]$ and $\beta \in (1,2)$, with each one requiring different conditions on $\Theta$. Naturally, this difference is due to $\bu$ being more singular for $\beta \in (1,2)$ than for $\beta \in (0,1]$. At a more technical level, this is caused by the different arguments required for estimating the component of $\bu$ in \eqref{velocity_field} where the integrand is restricted to the self-similar region, see \cref{lemmamain} and \cref{lemmamain2}. For $\beta \in (0,1]$, this is achieved thanks to the fact that the kernel $K_\beta$ from \eqref{velocity_field} satisfies $\| K_\beta \ast f \|_{L^q(B_\rho(0))} \lesssim \|f\|_{L^q(\R^2)}$ for every $f \in L^q(\R^2)$ and $1 < q < \infty$. In the case $\beta = 1$, this follows from the fact that $K_\beta$ is a Calder\'on-Zygmund operator, whereas for $\beta \in (0,1)$ this is a consequence of Young's convolution inequality together with $K_\beta$ being integrable near the origin. On the other hand, for $\beta \in (1,2)$, $K_\beta$ is neither a Calder\'on-Zygmund operator nor integrable near the origin. To circumvent this issue, we write $K_\beta(x) = \nabla^\perp (|x|^{-\beta})$ and integrate by parts in \eqref{velocity_field}, thus transferring one derivative to $\theta$, and hence to $\Theta$, when restricting $\theta$ to the self-similar region, see \eqref{eq3} and \eqref{eq4} below. For this reason, in comparison to the case $\beta \in (0,1]$, here we impose an additional growth assumption on the $L^r$ norm of $\nabla \Theta$ over increasing regions in the spatial domain, for suitable $r$, see \eqref{profile-hipothesis-derivada}.

Additionally, we note that \cref{thm:main2} partially recovers the aforementioned result established in \cite{CX15} concerning globally self-similar solutions of the gSQG equation for $\beta \in (0,1]$, and with weaker assumptions, since every globally self-similar solution is also locally self-similar. See \cref{rmk:lit:comp} below for the precise details. We also point out that the previously referenced result from \cite{garcia2022self} regarding the existence of a globally (hence locally) self-similar solution for the gSQG when $\beta \in (0,1)$ is not in contradiction with our results on nonexistence of nontrivial locally self-similar solutions, specifically as stated within \cref{thm:main2} and \cref{thm:main4}, \ref{cor:2:i}, below. Indeed, as we mentioned above, \cite{garcia2022self} presents a globally self-similar solution that possesses infinite energy, whereas our results concern solutions within a finite-energy class where local existence is established, namely $\theta \in \mathcal{C}([0, T); H^s(\R^2)) \cap L^\infty(0, T; L^1(\R^2))$, with $s > 1 +\beta$.

The remainder of this manuscript is organized as follows. In \cref{mainresults}, we present the statements of our main results. Their proofs are given in \cref{proof-of-results}. Finally, in \cref{appendix}, we show two crucial lemmas that are used in the proofs of the main results.

\section{Statements of the main results}\label{mainresults}

This section collects the statements of our main results. For the reasons described in the previous section, these are split between the two different cases $1 < \beta < 2$ and $0 < \beta \leq 1$ concerning the parameter $\beta$ in the constitutive relation between $\bu$ and $\theta$ in \eqref{DSQG}. 

Throughout the manuscript, we fix the standard notation of Sobolev spaces $H^s$, $s \in [0,\infty)$, and Lebesgue spaces $L^p$, $p \in [1,\infty]$. We also denote by $C$ a positive constant whose value may change from line to line. Moreover, we write $A \lesssim B$ to denote that $A \leq C B$ for some constant $C > 0$, and $A \sim B$ means that both $A \lesssim B$ and $B \lesssim A$ hold.

\subsection{The case $1< \beta <2$} 

	\begin{theorem}\label{thm:main}
	Fix $\beta \in (1,2)$. Suppose $\theta \in \mathcal{C}([0, T); H^s(\R^2)) \cap L^\infty(0, T; L^1(\R^2))$, with $s > 1 +\beta$, is a solution to the gSQG equation \eqref{DSQG} that is locally self-similar in a ball $B_\rho(x_0) \subset \R^2$, with scaling parameter $\alpha >-1$ and profile $\Theta \in \mathcal{C}^{1} (\R^2)$. Fix also $p \geq 1$, and suppose that for some $r \geq p+1$, $\gamma_1 \in [0,r(\beta - 1))$, and $\gamma_0 \geq 0$ with
		\begin{align}\label{assp:gamma0}
			\gamma_0 \leq \gamma_1 + r \quad \mbox{and} \quad \gamma_0  < \frac{(r-p)r}{p} \left( \beta - 1 - \frac{\gamma_1}{r} \right),
		\end{align}
	it holds
	\begin{equation}\label{profile-hipothesis-function}
		\int_{|y| \leq L} |\Theta (y)|^r dy \lesssim L^{\gamma_0}, 
	\end{equation}
	and
	\begin{equation}\label{profile-hipothesis-derivada}
		\int_{|y| \leq L} |\nabla \Theta (y)|^r dy \lesssim L^{\gamma_1}
	\end{equation}
	for all $L$ sufficiently large. Under these conditions, it follows that if $\alpha > \beta+\frac{2}{p} -1$ or $-1 < \alpha < \beta -1 + \frac{2-\gamma_0}{r}$ then $\Theta \equiv 0$. Moreover, if $\alpha \in \left[\beta -1 + \frac{2-\gamma_0}{r}, \beta -1 + \frac{2}{p} \right]$ then either $\Theta \equiv 0$ or $\Theta$ is a nontrivial profile and it satisfies 
	\begin{align}\label{growth-profile}
		\int_{|y| \leq L} |\Theta(y)|^p dy \sim  L^{2-p(1+\alpha -\beta)}
	\end{align}
	for all $L$ sufficiently large. 
\end{theorem}

\begin{remark}\label{remark-thm1}
	We note that, when $\gamma_0 = 0$, a slightly stronger result can be obtained in \cref{thm:main}. Namely, under $\gamma_0 = 0$, we have that $\Theta\equiv 0$ also in the case $\alpha =  \beta - 1 + 2/r$. Indeed, in this case, condition \eqref{profile-hipothesis-function} implies $\Theta\in L^r(\mathbb R^2)$ and one can use the same argument as in the beginning of the proof of \cite[Theorem 1.1]{CX15} to prove \eqref{limitL2}. Since this is the only instance in the proof of \cref{case2} where the strict inequality $\alpha < \beta - 1 + 2/r$ was required, the statement follows.	
\end{remark}

By verifying the assumptions of \cref{thm:main}, we may automatically exclude self-similar profiles with certain asymptotic behaviors, or guarantee a characterization of the $L^p$ norm as in \eqref{growth-profile} for possible types of blowup profiles. This is done in the following corollary.

\begin{corollary}\label{thm:main3}
Fix $\beta \in (1,2)$. Suppose $\theta \in \mathcal{C}([0, T); H^s(\R^2)) \cap L^\infty(0, T; L^1(\R^2))$, with $s > 1 +\beta$, is a solution to the gSQG equation \eqref{DSQG} that is locally self-similar in a ball $B_\rho(x_0) \subset \R^2$, with scaling parameter $\alpha >-1$ and profile $\Theta \in \mathcal{C}^{1} (\R^2)$. 
Then, the following statements hold:
\begin{enumerate}[label={(\roman*)}]
    \item\label{cor:1:i} If there exist some $\sigma_0 >0$ and $\sigma_1>0$ such that $|\Theta (y)| \lesssim |y|^{-\sigma_0}$ and $|\nabla \Theta (y)| \lesssim |y|^{-\sigma_1}$ for all $|y| \gg 1$, then $\Theta \equiv 0$ in $\R^2$.
	\item\label{cor:1:ii} Suppose that $|\Theta (y)| \gtrsim 1 $ for all $|y| \gg 1$, and that there exists a real number $0 \leq \sigma_1 < \beta -1$ such that 
$ |\nabla \Theta (y)| \lesssim |y|^{\sigma_1}$ for all $|y| \gg 1$. Then the values of $\alpha$
admitting nontrivial profiles belong to the interval $[\beta -2 -\sigma_1, \beta -1]$ and for each such $\alpha$ the corresponding profile $\Theta$ satisfies 
\begin{align*}
\int_{|y| \leq L} |\Theta(y)|^p dy \sim   L^{2-p(1+\alpha -\beta)}
\end{align*}
for every $p \in [1,\infty)$ and for all $L$ sufficiently large.
\end{enumerate}
\end{corollary}

\subsection{The case $0 < \beta \leq 1$}\label{0<beta<1}

The analogous versions of \cref{thm:main} and \cref{thm:main3} for the case $0 < \beta \leq 1$ are presented next. 

\begin{theorem}\label{thm:main2}
	Fix $\beta \in (0,1]$. Suppose $\theta \in \mathcal{C}([0, T); H^s(\R^2)) \cap L^\infty([0, T); L^1(\R^2))$, with $s > 1 + \beta$, is a solution to the gSQG equation \eqref{DSQG} that is locally self-similar in a ball $B_\rho(x_0) \subset \R^2$, with scaling parameter $\alpha > -1$ and profile $\Theta \in \mathcal{C}^\beta (\R^2)$. Fix also $p \geq 1$, and suppose that for some $r \geq p +1$ and $ \gamma \in [0,  \beta (r - p))$, it holds
	\begin{equation}\label{hipothesys-profile-thm2}
		\int_{|y| \leq L} |\Theta (y)|^r dy \lesssim L^\gamma
	\end{equation}
	for all $L$ sufficiently large. Under these conditions, it follows that if $\alpha > \beta -1 + \frac{2}{p}$ or $-1 < \alpha < \beta -1 + \frac{2-\gamma}{r}$ then $\Theta \equiv 0$. Moreover, if $\alpha \in \left[\beta -1 + \frac{2-\gamma}{r} , \beta -1 + \frac{2}{p} \right]$ then either $\Theta \equiv 0$ or $\Theta$ is a nontrivial profile and it satisfies 
	\begin{align}\label{growth-profile-thm2}
		\int_{|y| \leq L} |\Theta(y)|^p dy \sim L^{2-p(1+\alpha -\beta)}
	\end{align}
	for all $L$ sufficiently large.
\end{theorem}

\begin{remark}
	When \eqref{hipothesys-profile-thm2} holds with $\gamma = 0$, then we have that $\Theta\equiv 0$ also in the case $ \alpha=  \beta - 1 + 2/r$. The justification is the same as in \cref{remark-thm1}. 
\end{remark}

\begin{corollary}\label{thm:main4}
Fix $\beta \in (0,1]$. Suppose $\theta \in \mathcal{C}([0, T); H^s(\R^2)) \cap L^\infty(0, T; L^1(\R^2))$, with $s > 1 +\beta$, is a locally self-similar solution to the gSQG equation that is locally self-similar in a ball $B_\rho(x_0) \subset \R^2$, with scaling parameter $\alpha >-1$ and profile $\Theta \in \mathcal{C}^{\beta} (\R^2)$.
Then, the following statements hold:
\begin{enumerate}[label={(\roman*)}]
    \item\label{cor:2:i} If there exists some $\sigma > 0$ such that $|\Theta (y)| \lesssim |y|^{-\sigma}$ for all $|y| \gg 1$, then no locally self-similar blowup occurs, i.e., $\Theta \equiv 0$ in $\R^2$.
	\item\label{cor:2:ii} If there exists some $\sigma \in (0,\beta) $ such that $1 \lesssim |\Theta (y)| \lesssim |y|^{\sigma}$  for all $|y| \gg 1$, then the values of $\alpha$
admitting nontrivial profiles belong to $[\beta -1 - \sigma,\beta -1] $, and for each such $\alpha$ the corresponding profile $\Theta$ satisfies 
\begin{eqnarray*}
\int_{|y| \leq L} |\Theta(y)|^p dy \sim  L^{2-p(1+\alpha -\beta)},
\end{eqnarray*}
for every $p \in [1,\infty)$ and for all $L$ sufficiently large. 

\end{enumerate}
\end{corollary}

\begin{remark}\label{rmk:lit:comp}
As mentioned in \cref{sec:intro}, \cref{thm:main2} generalizes the results proved in \cite{Xue16} regarding the SQG equation ($\beta =1$) to the generalized SQG equation for all $\beta \in (0,1]$. Furthermore, our assumptions on the parameters $r$ and $\gamma$ for \eqref{hipothesys-profile-thm2} to hold, namely $r > p$ and $\gamma \in [0, r +2)$, are weaker than those in \cite[Theorem 1.1]{Xue16}, where it is assumed that $r \geq p+1$ and $\gamma \in [0, r-p)$.

Moreover, for $\beta \in (0,1]$, \cref{thm:main2} partially recovers the result established in \cite{CX15} concerning globally self-similar solutions of the gSQG equation, with weaker assumptions, since any globally self-similar solution is also locally self-similar. More precisely, we recover the particular case of \cite[Theorem 1.1]{CX15} when $p+1<2/(\gamma-1)$ and $p+1\leq q<2/(\gamma-1)$ (here $p, q$ and $\gamma$ are as in \cite{CX15}, which correspond to $p$, $r$, and $2-\beta$ in our notation, respectively).
\end{remark}

\section{Proofs}\label{proof-of-results}

Before turning to the proofs of our main results, let us provide a brief summary of the underlying ideas. Similarly as in \cite{CX15,Xue16}, our arguments rely crucially on the following local $L^p$ equality satisfied by any solution $\theta \in \mathcal{C}([0,T); H^{s}(\R^2)) \cap L^\infty(0, T; L^p(\R^2))$, with $s > 1+\beta$, of \eqref{DSQG}. Namely, for fixed $0 <t_1< t_2 < T$ and $p \in [1, \infty)$,
\begin{multline}\label{eq:loc:Lp}
	\int_{\R^2} |\theta(x,t_2)|^p \eta (x,t_2) dx - \int_{\R^2} |\theta(x,t_1)|^p \eta (x,t_1) dx \\
	= \int_{t_1}^{t_2}\int_{\R^2} |\theta(x,t)|^p \partial_t\eta (x,t) dx dt + \int_{t_1}^{t_2}\int_{\R^2} |\theta(x,t)|^p (\bu \cdot \nabla) \eta (x,t) dx dt,
\end{multline}
for every smooth and compactly supported test function $\eta$ on $[0, \infty) \times \R^2$, i.e. $\eta \in \mathcal{C}^{\infty}_{c}([0, \infty) \times \R^2, \R)$. Its proof follows by taking a mollification $\theta_\varepsilon = \rho_\varepsilon \ast \theta$ of $\theta$, for a standard mollifier $\rho_\varepsilon$, and noting that $\theta_\varepsilon$ satisfies the equation
\begin{align}\label{eq:theta:ve}
	\partial_t \theta_\varepsilon + (\bu_\varepsilon \cdot \nabla ) \theta_\varepsilon = (\bu_\varepsilon \cdot \nabla ) \theta_\varepsilon - \rho_\varepsilon \ast [(\bu \cdot \nabla )\theta].
\end{align}
Then, multiplying \eqref{eq:theta:ve} by $\theta_\varepsilon |\theta_\varepsilon|^{p-2}\ \psi_\varepsilon(x,t)$ for some suitable test function $\psi_\varepsilon$ such that $\supp \psi_\varepsilon \subseteq \supp \theta_\varepsilon$, and carefully taking the limit as $\varepsilon \to 0$ leads to \eqref{eq:loc:Lp}.

The proof of \cref{thm:main} is divided into three cases, each one corresponding to one of the ranges of $\alpha$ described in the statement. In the first two cases, corresponding to $\alpha$ sufficiently large or sufficiently small, the goal consists in showing that 
\begin{equation}\label{Lp:Theta:sig}
	\int_{|y| \lesssim L} |\Theta (y) |^p dy \lesssim L^{\sigma} \quad \mbox{for some } \sigma<0.
\end{equation}
Clearly, taking the limit as $L \to \infty$ then implies that $\Theta \equiv 0$. When $\alpha$ is sufficiently large, \eqref{Lp:Theta:sig} follows directly from the local self-similarity condition, \eqref{ss:x0:T}, combined with the maximum principle satisfied by the solution $\theta$. On the other hand, with $\alpha$ small enough, an inequality as in \eqref{Lp:Theta:sig} is achieved by establishing a fundamental local $L^p$ inequality from the local equality \eqref{eq:loc:Lp}, see \eqref{maininequality2} below. This is derived by suitably employing cut-off functions to split the velocity field $\bu$ into its restrictions to the self-similar region and the corresponding exterior. With the help of assumptions \eqref{profile-hipothesis-function}, \eqref{profile-hipothesis-derivada}, and \cref{lemmamain}, we then estimate the terms on the right-hand side of \eqref{maininequality2} to yield an upper bound for $\int_{|y| \lesssim L} |\Theta (y) |^p dy$. Next, we redo the estimates by using this new upper bound and bootstrap on this argument until we eventually arrive at an upper bound as in 
\eqref{Lp:Theta:sig} with a negative power of $L$.

For the last and intermediate range of $\alpha$, we must show that every nontrivial profile $\Theta$ satisfies the lower and upper bounds implied by \eqref{growth-profile}. The upper bound is guaranteed from the estimate derived for the first range of $\alpha$, whereas for the lower bound we proceed by contradiction. Namely, assuming that such lower estimate does not hold, it follows that $\Theta$ must satisfy the same local $L^p$ inequality as in the second case, \eqref{maininequality2}. Proceeding with a similar analysis from this case, we then arrive at the contradiction that $\Theta \equiv 0$.

The proof of \cref{thm:main3} follows by choosing appropriate parameters $p,r,\gamma_0$, and $\gamma_1$ so that the assumptions of \cref{thm:main} are verified under the conditions imposed on the profile $\Theta$ in each of the items \ref{cor:1:i} and \ref{cor:1:ii}. Finally, the proofs of \cref{thm:main2} and \cref{thm:main4} are obtained under the same line of reasoning as in the previous two results. The central difference lies on the use of \cref{lemmamain2} below instead of \cref{lemmamain}.

\subsection{Proof of \cref{thm:main}}

Without loss of generality, we may assume $x_0=0$. The proof is divided into three different cases, each corresponding to a particular range for $\alpha$ within the interval $(-1, \infty)$.
\vspace{0.2cm}
\noindent
\begin{case}\label{case1} Suppose $\alpha > \beta+\frac{2}{p} -1$. In this case, we show that $\Theta \equiv 0$ in $\R^2$. 
	
	Fix $t \in [0,T)$ and denote $L = \rho (T-t)^\frac{-1}{1+\alpha}$. Invoking the local self-similarity of $\theta$, namely \eqref{ss:x0:T}, and changing variables, it follows that
	\begin{align}\label{eq-initial}
		\int_{|x| \leq \rho} |\theta (x,t)|^p dx 
		= \frac{1}{(T-t)^{\frac{p(1+\alpha -\beta)}{1+\alpha}}} \int_{|x| \leq \rho} \left| \Theta \left(\frac{x}{(T-t)^{\frac{1}{1+\alpha}}} \right) \right|^p dx
		= C L^{p(1+\alpha -\beta) -2} \int_{|y| \leq L} |\Theta(y)|^p dy.
	\end{align}
	
	Since $s >1$, it follows by Sobolev embedding that $\theta(0) \in H^s(\R^2) \subset L^{\tp}(\R^2)$ for every $\tp \geq 2$. This implies that
	\begin{align}\label{eq:max:princ}
		\|\theta(t)\|_{L^{\tp}} =\|\theta(0)\|_{L^{\tp}} \quad \mbox{for all }  t \in [0,T) \, \mbox{ and } \, \tp \geq 2,
	\end{align}
	see e.g. \cite[Theorem 3.3]{Resnick05}. Thus, by H\"{o}lder's inequality, it follows that for all $p \in [1, \infty)$ and $\tp \geq \max\{2,p\}$, we have
	\begin{align*}
		\int_{|x| \leq \rho} |\theta (x,t)|^p dx \leq C \| \theta (0)\|_{L^{\tp}} \quad \mbox{ for all } t \in [0,T). 
	\end{align*}
	Hence, we obtain from \eqref{eq-initial} that
	\begin{align}\label{eq1step2}
		\int_{|y| \leq L} |\Theta(y)|^p dy \leq CL^{2- p(1+\alpha -\beta)}.
	\end{align}
	Since $2 - p(1+\alpha -\beta) < 0$, taking the limit as $t \rightarrow T$ in \eqref{eq1step2}, which implies $L \to \infty$, we deduce that $\Theta \equiv 0$ in $\R^2$.	
\end{case}	

\begin{case}\label{case2} Let us now suppose that $-1 < \alpha < \beta -1 + \frac{2-\gamma_0}{r}$. Here we once again show that $\Theta \equiv 0$ in $\R^2$.
	
	Take cut-off functions $\phi_{\frac{\rho}{4}}, \phi_\rho \in \cC^\infty(\R^2)$ with $0 \leq \phi_{\frac{\rho}{4}}, \phi_\rho \leq 1$, $\phi_{\frac{\rho}{4}} \equiv 1$ in $B_{\rho/8}(0)$, $\phi_{\frac{\rho}{4}} \equiv 0$ in $B_{\rho/4}^c(0)$, and $\phi_\rho \equiv 1$ in $B_{\rho/2}(0)$, $\phi_\rho \equiv 0$ in $B_\rho^c(0)$. Fix $t_1, t_2 \in [0, T)$. From \eqref{eq:loc:Lp}, we have in particular that	
	\begin{align}\label{eq2}
		\int_{\R^2} |\theta(x,t_2)|^p \phi_{\frac{\rho}{4}}(x) dx - \int_{\R^2} |\theta(x,t_1)|^p \phi_{\frac{\rho}{4}}(x) dx  =\int_{t_1}^{t_2} \int_{\R^2} (\bu(x,t) \cdot \nabla \phi_{\frac{\rho}{4}}(x) )  |\theta(x,t)|^p  dx \,dt.
	\end{align}
	We proceed to analyze each term in \eqref{eq2}, starting with the first two terms on the left-hand side. By the local self-similarity of $\theta$, \eqref{ss:x0:T}, it follows that for $i = 1,2$
	\begin{align}\label{eq_selfsimilar}
		\int_{\R^2} |\theta(x, t_i)|^p \phi_{\frac{\rho}{4}} (x) dx 
		&= \frac{1}{(T-t_i)^{\frac{p(1+\alpha -\beta)}{1+\alpha}}} \int_{\R^2}   \bigg| \Theta \bigg(\frac{x}{(T-t_i)^{\frac{1}{1+\alpha}}}\bigg) \bigg|^p \phi_{\frac{\rho}{4}} (x) dx \notag \\
		&= \frac{1}{(T-t_i)^{\frac{p(1+\alpha -\beta)-2}{1+\alpha}}} \int_{\R^2} |\Theta (y) |^p \phi_{\frac{\rho}{4}} (y(T-t_i)^{\frac{1}{1+\alpha}}) dy \notag \\
		&= l_i^{p(1 + \alpha - \beta) -2} \int_{|y| \leq \frac{\rho}{4}l_i} |\Theta (y)|^p \phi_{\frac{\rho}{4}} (y l_{i}^{-1}) dy,
	\end{align}
	where $l_i = (T-t_i)^{- \frac{1}{1+\alpha}}$, $i =1, 2$. 
	
	To analyze the term in the right-hand side of \eqref{eq2}, we first decompose the velocity field $\bu$ into a term involving the self-similarity region and another one outside of it. More precisely, recalling \eqref{velocity_field} and \eqref{kernel}, we have
	\begin{align}\label{eq3}
		\bu(x,t) 
		&= C_\beta P.V. \int_{\R^2} K_\beta (x-y) \theta(y,t)\phi_\rho (y) dy +  C_\beta P.V. \int_{\R^2} K_\beta (x-y) \theta(y,t)(1 - \phi_\rho(y)) dy \notag \\
		&= C_\beta P.V. \int_{\R^2} \frac{(x-y)^{\perp}}{|x-y|^{2+\beta}} \theta(y,t) \phi_\rho(y) dy + C_\beta P.V. \int_{\R^2} K_\beta (x-y) \theta(y,t)(1 - \phi_\rho(y)) dy \notag \\
		&= \frac{1}{\beta}C_\beta P.V. \int_{\R^2} \nabla_{y}^{\perp}\left( \frac{1}{|x-y|^{\beta}}\right) \theta(y,t) \phi_\rho (y) dy + C_\beta P.V. \int_{\R^2} K_\beta (x-y) \theta(y,t)(1 - \phi_\rho(y)) dy \notag \\
		&= - \frac{1}{\beta} C_\beta P.V. \int_{\R^2} \frac{1}{|x-y|^{\beta}} \nabla^\perp \theta(y,t) \phi_\rho(y) dy - \frac{1}{\beta} C_\beta P.V. \int_{\R^2} \frac{1}{|x-y|^{\beta}} \theta(y,t) \nabla^\perp \phi_\rho (y) dy   \notag \\
		&\qquad + C_\beta P.V. \int_{\R^2} K_\beta (x-y) \theta(y,t)(1 - \phi_\rho(y)) dy \notag \\
		&=: {\bu}^{(1)}(x,t) + {\bu}^{(2)}(x,t) + \bu^{(3)}(x,t),
	\end{align}
	where the second to last equality follows by integration by parts. We now analyze each of these terms. By the local self-similarity of $\theta$, \eqref{ss:x0:T}, we get
	\begin{align}\label{eq4}
		\bu^{(1)}(x,t) &= - \frac{1}{\beta}C_\beta  P.V.  \int_{\R^2} \frac{1}{|x-y|^{\beta}} \nabla^\perp \theta(y,t) \phi_\rho(y) dy \notag \\
		&= - \frac{C_\beta}{\beta(T-t)^{\frac{2+\alpha -\beta}{1+\alpha}}}  P.V. \int_{\R^2} \frac{1}{|x-y|^{\beta}} \nabla^\perp \Theta \left( \frac{y}{(T-t)^{\frac{1}{1+\alpha}}}\right) \phi_\rho (y) dy \notag \\
		&= -  \frac{C_\beta}{\beta(T-t)^{\frac{\alpha -\beta}{1+\alpha}}}  P.V. \int_{\R^2} \frac{1}{|x-(T-t)^{\frac{1}{1+\alpha}}y|^{\beta}} \nabla^\perp \Theta (y) \phi_\rho (y(T-t)^{\frac{1}{1+\alpha}}) dy \notag \\
		&= -  \frac{C_\beta}{\beta(T-t)^{\frac{\alpha }{1+\alpha}}}  P.V. \int_{\R^2} \frac{1}{|(T-t)^{\frac{-1}{1+\alpha}}x -y|^{\beta}} \nabla^\perp \Theta (y) \phi_\rho (y(T-t)^{\frac{1}{1+\alpha}}) dy \notag \\
		&= - \frac{C_\beta}{\beta(T-t)^{\frac{\alpha }{1+\alpha}}} \v \left( \frac{x}{(T-t)^{\frac{1}{1+\alpha}}}, t \right),
	\end{align}
	where
	\begin{align}\label{U1(y,t)}
		\v(x,t) :=  P.V. \int_{\R^2} \frac{1}{|x -y|^{\beta}} \nabla^\perp \Theta (y) \phi_\rho (y(T-t)^{\frac{1}{1+\alpha}}) dy.
	\end{align}
	
	Next, we analyze $\bu^{(2)}(x,t)$. Note that due to the presence of $\nabla \phi_{\frac{\rho}{4}}(x)$ in the right-hand side of \eqref{eq2}, it suffices to  consider $x \in \R^2$ with ${\rho}/{8} \leq |x| \leq \rho/4$. Then, since for each such $x$ we have $|x-y| \geq \frac{|y|}{2}$ for every $|y| \geq \rho/2$, it follows that
	\begin{align}\label{u2}
		|\bu^{(2)}(x,t)| 
		&\leq C \int_{\frac{\rho}{2} \leq |y| \leq \rho}    \frac{1}{|x-y|^{\beta}} |\theta(y,t)| |\nabla^\perp \phi_\rho (y)|dy \notag \\
		&\leq  C  \int_{|y| \geq \frac{\rho}{2}} \frac{|\theta(y,t)|}{|y|^{\beta}}  dy \notag \\
		&\leq C \|\theta\|_{L^\infty(0,T; L^2(\R^2))} 
		\leq C \|\theta(0)\|_{L^2},
	\end{align}
	where in the last line we applied H\"older's inequality and \eqref{eq:max:princ} with $\tp = 2$.
	
	Finally, for the last term in \eqref{eq3}, $\bu^{(3)}(x,t)$, we proceed similarly as was done for $\bu^{(2)}(x,t)$ in \eqref{u2} and obtain that
	\begin{align}\label{eq6}
		|\bu^{(3)}(x,t)| &=  C_\beta\int_{|y| \geq \rho/2} \frac{1}{|x-y|^{\beta+1}} |\theta(y,t)|(1 - \phi_\rho(y)) dy \notag\\
		&\leq C \int_{|y| \geq \frac{\rho}{2}} \frac{|\theta(y,t)| }{|y|^{\beta +1}} dy 
		\leq C \|\theta(0)\|_{L^2}.
	\end{align}
	
	From \eqref{eq4}, \eqref{u2} and  \eqref{eq6}, we may then estimate the term in the right-hand side of \eqref{eq2} as
	\begin{align*}
		&\bigg| \int_{t_1}^{t_2}  \int_{\R^2}  |\theta(x,t)|^p  (\bu(x,t) \cdot \nabla \phi_{\frac{\rho}{4}}(x))  dx \,dt \bigg| \notag \\
		&\qquad\leq \int_{t_1}^{t_2} \int_{\R^2} |\bu_1(x,t)| |\nabla \phi_{\frac{\rho}{4}}(x)|  |\theta(x,t)|^p  dx \,dt + \int_{t_1}^{t_2} \int_{\R^2} |\bu_2(x,t)| |\nabla \phi_{\frac{\rho}{4}}
		(x)|  |\theta(x,t)|^p  dx \,dt \notag \\
		&\qquad\qquad+ \int_{t_1}^{t_2} \int_{\R^2} |\bu_3(x,t)| |\nabla \phi_{\frac{\rho}{4}}(x)|  |\theta(x,t)|^p  dx \,dt	
	\end{align*} 
	\begin{align}\label{bb}
		&\qquad\leq C  \int_{t_1}^{t_2} \frac{1}{(T-t)^{\frac{\alpha + p(1+\alpha -\beta)}{1+\alpha}}} \int_{\R^2} \left| \v \left( \frac{x}{(T-t)^{\frac{1}{1+\alpha}}}, t \right) \right| \left| \Theta \left(\frac{x}{(T-t)^{\frac{1}{1+\alpha}}}\right)\right|^p |\nabla \phi_{\frac{\rho}{4}} (x)| dx \,dt \notag \\
		&\qquad\qquad+  C \int_{t_1}^{t_2} \frac{1}{(T-t)^{\frac{p(1+\alpha -\beta)}{1+\alpha}}} \int_{\R^2} \bigg| \Theta \left(\frac{x}{(T-t)^{\frac{1}{1+\alpha}}}\right) \bigg|^p  |\nabla \phi_{\frac{\rho}{4}}(x)| dx \,dt \notag \\
		&\qquad\leq C \int_{t_1}^{t_2} \frac{1}{(T-t)^{\frac{ \alpha -2 + p(1+\alpha -\beta)}{1+\alpha}}} \int_{\R^2} |\v \left(y, t \right)| | \Theta(y) |^p |\nabla \phi_{\frac{\rho}{4}} (y(T-t)^{\frac{1}{1+\alpha}})| dy \,dt \notag \\
		&\qquad\qquad+ C \int_{t_1}^{t_2} \frac{1}{(T-t)^{\frac{p(1+\alpha -\beta)-2}{1+\alpha}}} \int_{\R^2}  |\Theta(y)|^p  |\nabla \phi_{\frac{\rho}{4}}(y(T-t)^{\frac{1}{1+\alpha}})| dy \,dt.
	\end{align}
	
	In view of the support of $\nabla \phi_{\frac{\rho}{4}}$, we may restrict the integrands in \eqref{bb} to $(y,t) \in \R^2 \times [t_1,t_2]$ such that $\rho/8 \leq |y| (T-t)^{\frac{1}{1+\alpha}} \leq \rho/4$. In particular, each such $y$ satisfies $\rho l_1/8 \leq |y| \leq \rho l_2/4$, where we recall that $l_i = (T - t_i)^{-\frac{1}{1+\alpha}}$, $i=1,2$. Then, for each fixed $y \in \R^2$ with $\rho l_1/8 \leq |y| \leq \rho l_2/4$, we define the set
	\begin{align}\label{ball-defi}
		\Ay := \left\{ t \in [t_1, t_2] \,:\, \frac{\rho}{8} \frac{1}{|y|} \leq (T -t)^{\frac{1}{1+\alpha}} \leq \frac{\rho}{4} \frac{1}{|y|}\right\}.
	\end{align}
	After rearrangement, it is easy to see that 
	\[\Ay\subset \left[T - \left(\frac{\rho}{4|y|}\right)^{1 + \alpha}, T - \left(\frac{\rho}{8|y|}\right)^{1 + \alpha}\right],\]
	so that its length satisfies $|\Ay| \leq c_{\alpha,\rho}/|y|^{1+\alpha}$. Thus, denoting by $\ind_{\Ay}$ the indicator function of the set $\Ay$, it follows from \eqref{bb} that	
	\begin{align}\label{eq7}
		&\bigg| \int_{t_1}^{t_2}  \int_{\R^2}  |\theta(x,t)|^p  (\bu(x,t) \cdot \nabla \phi_{\frac{\rho}{4}}(x))  dx \, dt \bigg| \notag \\
		&\qquad\leq C \int_{t_1}^{t_2} \int_{\frac{\rho}{8} l_1 \leq |y| \leq \frac{\rho}{4}l_2} \frac{|\v\left(y, t \right)| | \Theta(y) |^p}{|y|^{2 -\alpha - p(1+\alpha -\beta)}} \ind_{\Ay}(t) dy \, dt \notag \\
		&\qquad\qquad\qquad + C \int_{t_1}^{t_2}\int_{\frac{\rho}{8} l_1 \leq |y| \leq \frac{\rho}{4}l_2} \frac{|\Theta(y)|^p}{|y|^{2-p(1+\alpha -\beta)}}   \ind_{\Ay}(t)  dy \, dt \notag \\
		&\qquad\leq C \int_{\frac{\rho}{8} l_1 \leq |y| \leq \frac{\rho}{4}l_2} \frac{| \Theta(y) |^p}{|y|^{2 -\alpha - p(1+\alpha -\beta)}}  \int_{t_1}^{t_2} |\v\left(y, t \right)| \ind_{\Ay}(t) dt \, dy \notag \\
		&\qquad\qquad\qquad+ C\int_{\frac{\rho}{8} l_1 \leq |y| \leq \frac{\rho}{4}l_2} \frac{|\Theta(y)|^p}{|y|^{2-p(1+\alpha -\beta)}}   \int_{t_1}^{t_2} \ind_{\Ay}(t) dt \, dy \notag \\
		&\qquad\leq C \int_{\frac{\rho}{8} l_1 \leq |y| \leq \frac{\rho}{4}l_2} \frac{|\vt (y)| | \Theta(y) |^p}{|y|^{2 -\alpha - p(1+\alpha -\beta)}} dy  + C \int_{\frac{\rho}{8} l_1 \leq |y| \leq \frac{\rho}{4}l_2} \frac{|\Theta(y)|^p}{|y|^{3 + \alpha -p(1+\alpha -\beta)}} dy ,
	\end{align}
	where 
	\begin{align}\label{vt}
		\vt (y)&:= \int_{t_1}^{t_2} |\v(y,t)| \ind_{\Ay}(t) dt \notag \\
		&= \int_{t_1}^{t_2} \bigg| \int_{\R^2} 
		\frac{1}{|y -z|^{\beta}} \nabla^\perp \Theta (z) \phi_\rho (z(T-t)^{\frac{1}{1+\alpha}}) dz \bigg| \ind_{\Ay}(t) dt.
	\end{align}
	Plugging \eqref{eq7} into \eqref{eq2} and recalling \eqref{eq_selfsimilar}, yields
	\begin{multline}\label{maininequality} \bigg| l_2^{p(1 + \alpha - \beta) -2} \int_{\R^2} |\Theta (y)|^p \phi_{\frac{\rho}{4}} (y l_{2}^{-1}) dy - l_1^{p(1 + \alpha - \beta) -2} \int_{\R^2} |\Theta (y)|^p \phi_{\frac{\rho}{4}} (y l_{1}^{-1}) dy \bigg| \\
		\leq C \int_{\frac{\rho}{8} l_1 \leq |y| \leq \frac{\rho}{4}l_2} \frac{|\vt(y)| | \Theta(y) |^p}{|y|^{2 -\alpha - p(1+\alpha -\beta)}} dy  + C \int_{\frac{\rho}{8} l_1 \leq |y| \leq \frac{\rho}{4}l_2} \frac{|\Theta(y)|^p}{|y|^{3 + \alpha -p(1+\alpha -\beta)}} dy.
	\end{multline}
	Note that, by Hölder's inequality and assumption \eqref{profile-hipothesis-function}, it follows that
	\begin{align*}
		l_2^{p(1 + \alpha - \beta) -2} \int_{\R^2} |\Theta (y)|^p \phi_{\frac{\rho}{4}} (y l_{2}^{-1}) dy 
		&\leq C l_2^{p(1 + \alpha - \beta) -2} \bigg(\int_{|y| \leq \frac{\rho}{4}l_2} |\Theta (y)|^r dy\bigg)^{\frac{p}{r}} l_{2}^{2\left(1 - \frac{p}{r}\right)} \notag \\
		&\leq C l_{2}^{p(1 + \alpha -\beta) + (\gamma_0 - 2)\frac{p}{r}}.
	\end{align*}
	Since, by the current assumption on $\alpha$, we have $(1 + \alpha -\beta) + (\gamma_0 - 2)/r < 0$, then
	\begin{align}\label{limitL2}
		l_2^{p(1 + \alpha - \beta) -2} \int_{\R^2} |\Theta (y)|^p \phi_{\frac{\rho}{4}} (y l_{2}^{-1}) dy \rightarrow 0 \quad \mbox{ as } l_2 \rightarrow \infty.
	\end{align}
	Thus, denoting $L := \frac{\rho}{8}l_1$ and taking the limit in \eqref{maininequality} as $t_2 \to T$, so that $l_2 \to \infty$, we obtain
	\begin{align}\label{maininequality2}
		\frac{1}{L^{2- p(1 + \alpha - \beta)}} \int_{|y| \leq L} |\Theta (y)|^p dy  
		\leq C \int_{|y| \geq L} \frac{|\vt (y)| | \Theta(y) |^p}{|y|^{2 -\alpha - p(1+\alpha -\beta)}} dy  + C \int_{|y| \geq L} \frac{|\Theta(y)|^p}{|y|^{3 + \alpha -p(1+\alpha -\beta)}} dy.
	\end{align}
	In what follows, we always assume that $L$ is sufficiently large (equivalently, $t_1$ is sufficiently close to $T$), so that assumptions \eqref{profile-hipothesis-function} and \eqref{profile-hipothesis-derivada} can be applied.
	
	We now further estimate each of the terms on the right-hand side of \eqref{maininequality2} by splitting the integrals according to a dyadic decomposition. For the first term, we make use of \cref{lemmamain} below, which yields a control on the $L^r$ norm of the function $\vt$ on a dyadic shell under assumption \eqref{profile-hipothesis-derivada}. We obtain
	\begin{align}\label{eq9}
		&\int_{|y| \geq L} \frac{|\vt (y)| | \Theta(y) |^p}{|y|^{2 -\alpha - p(1+\alpha -\beta)}} dy \notag \\
		&\qquad\leq \sum_{k=0}^{\infty} \frac{1}{(2^k L)^{2 - \alpha - p(1 + \alpha - \beta)}} \int_{2^k L \leq |y| \leq 2^{k+1} L}  |\vt (y)| |\Theta(y)|^p dy \notag \\
		&\qquad\leq C \sum_{k=0}^{\infty} \frac{1}{(2^k L)^{2 - \alpha - p(1 + \alpha - \beta)}} \left( \int_{|y| \sim 2^kL} |\vt (y)|^r dy \right)^{\frac{1}{r}} \left( \int_{|y| \sim 2^kL} |\Theta (y)|^r dy \right)^{\frac{p}{r}}  (2^kL)^{2\left(1 - \frac{p+1}{r}\right)} \notag \\
		&\qquad\leq C\sum_{k=0}^{\infty} \frac{1}{(2^k L)^{2 - \alpha - p(1 + \alpha - \beta)}} (2^k L)^{1-\alpha - \beta + \frac{\gamma_1}{r}} (2^k L)^{\gamma_0\frac{p}{r}}  (2^k L)^{ 2 - \frac{2p}{r} - \frac{2}{r}} \notag \\
		&\qquad\leq C\sum_{k=0}^{\infty} (2^kL)^{p(1+\alpha -\beta) -2 - \beta + 3 + \frac{\gamma_1 - 2}{r} + \frac{(\gamma_0 -2)p}{r}} \notag \\
		&\qquad\leq C L^{p(1+\alpha -\beta) -2 - \beta + 3 + \frac{\gamma_1 - 2}{r} + \frac{(\gamma_0 -2)p}{r}},
	\end{align}
	where in the second inequality we used that $r \geq p+1$ and applied H\"older's inequality, and in the last inequality we used the hypotheses that $\alpha < \beta -1 + \frac{2-\gamma_0}{r}$ and $\gamma_1 < r(\beta-1)$. 
	
	For the second term in the right-hand side of \eqref{maininequality2}, applying again the dyadic decomposition together with Hölder's inequality, yields
	\begin{align}\label{eq10}
		\int_{|y| \geq L} \frac{|\Theta (y)|^p}{|y|^{3 + \alpha - p(1+\alpha - \beta})} dy 
		&\leq  \sum_{k=0}^{\infty} \frac{1}{(2^k L)^{3 +\alpha - p(1 + \alpha - \beta)}} \int_{|y| \sim 2^kL} |\Theta(y)|^p dy \notag \\
		&\leq  C \sum_{k=0}^{\infty} \frac{1}{(2^k L)^{3 +\alpha - p(1 + \alpha - \beta)}} \bigg(\int_{|y| \sim 2^kL} |\Theta(y)|^r dy\bigg)^{\frac{p}{r}} (2^k L)^{2\left(1-\frac{p}{r}\right)} \notag \\
		&\leq C \sum_{k=0}^{\infty} (2^k L)^{p(1+\alpha -\beta) -3 - \alpha} (2^k L)^{\gamma_0\frac{p}{r}} (2^k L)^{2\left(1-\frac{p}{r}\right)} \notag \\
		&\leq C \sum_{k=0}^{\infty} (2^k L)^{p(1 +\alpha -\beta) -2 -\alpha +1 + (\gamma_0 -2)\frac{p}{r}} \notag \\
		&\leq  C L^{p(1 +\alpha -\beta) -2 + 1-\alpha  + \frac{(\gamma_0 -2)p}{r}},
	\end{align}
	where we used that $-1 < \alpha < \beta -1 + \frac{2-\gamma_0}{r}$. Combining \eqref{eq9} and \eqref{eq10} with \eqref{maininequality2}, we deduce that
	\begin{align}\label{eqa_0}
		\int_{|y| \leq L} |\Theta (y)|^p dy  &\leq  C L^{3-\beta + \frac{(\gamma_1-2)}{r} + \frac{(\gamma_0 -2)p}{r}} + C L^{1 -\alpha + \frac{(\gamma_0 -2)p}{r}} \leq C L^{a_0},
	\end{align}
	where
	\begin{eqnarray}\label{a_0defi}
		a_0:= \max \left\{ 1 -\alpha + \frac{(\gamma_0 -2)p}{r}, 3 - \beta + \frac{(\gamma_1-2)}{r} + \frac{(\gamma_0 -2)p}{r}\right\}.
	\end{eqnarray}
	
	Note that, if $a_0 <0$, we conclude that $\Theta \equiv 0$ on $\R^2$, i.e., no locally self-similar blowup occurs and the proof is finished. Otherwise, if $a_0 \geq 0$, we improve the estimates in \eqref{eq9} and \eqref{eq10} by making use of the new upper bound in \eqref{eqa_0}. In particular, for the first term in the right-hand side of \eqref{maininequality2}, we leverage \eqref{eqa_0} via a suitable interpolation inequality. Firstly, for simplicity of notation, we denote $q:= rp/(r-1)$, and write the given assumptions on $\gamma_0$ and $\gamma_1$ in terms of $q$ as
	\begin{align}\label{cond:gamma1:q}
		0 \leq \gamma_1 < r \left( \beta -1 - 2 \left( 1 - \frac{p}{q} \right) \right) +2, 
	\end{align}
	and $\gamma_0 \in [0, \gamma_1 + r]$ with
	\begin{align} \label{cond:gamma0:q}
		\gamma_0 < 
		\frac{(r-p)q}{(q-p)p} \left( \beta - 1 
		- \frac{\gamma_1}{r} \right).
	\end{align}
	Note that $\frac{1}{r} + \frac{p}{q} = 1$ and $p < q \leq r$. Then, by interpolation, we have
	\begin{align}
		\int_{|y| \leq L} |\Theta(y)|^{q} dy &\leq \left( \int_{|y| \leq L} |\Theta (y)|^p dy \right)^{\delta} \left( \int_{|y| \leq L} |\Theta(y)|^r dy \right)^{1 -\delta}  \label{eq11a}\\
		&\leq C L^{a_0\delta + (1-\delta)\gamma_0}, \quad  \textrm{with }  \delta := \frac{r-q}{r-p} \in [0,1). \label{eq11}
	\end{align}
	
	Next, employing once again the dyadic decomposition and H\"older's inequality, we derive via \eqref{eqa_0}, \eqref{eq11}, and \cref{lemmamain} that
	\begin{align}\label{ff}
		&\int_{|y| \geq L} \frac{|\vt (y)| | \Theta(y) |^p}{|y|^{2 -\alpha - p(1+\alpha -\beta)}} dy \notag \\
		&\qquad\leq  C \sum_{k=0}^{\infty} \frac{1}{(2^k L)^{2 - \alpha - p(1 + \alpha - \beta)}} \left( \int_{|y| \sim 2^kL} |\Theta (y)|^{q} dy \right)^{\frac{p}{q}} \left( \int_{|y| \sim 2^kL} |\vt (y)|^{r} dy \right)^{\frac{1}{r}} 
		\notag \\
		&\qquad\leq C \sum_{k=0}^{\infty} \frac{1}{(2^k L)^{2 - \alpha - p(1 + \alpha - \beta)}} (2^k L)^{\frac{p}{q}\left(a_0 \delta + (1-\delta)\gamma_0\right)} (2^k L)^{\frac{\gamma_1}{r} + 1 -\alpha - \beta}
		\notag \\
		&\qquad\leq C \sum_{k=0}^{\infty} (2^k L)^{ p(1 + \alpha -\beta) -2 + a_0 + \frac{p}{q}(1-\delta)(\gamma_0 - a_0) + 
		{\gamma_1}{r}  + 1 - \beta - a_0 \left( 1 - \frac{p}{q}\right) }
	\notag \\
	&\qquad\leq C \sum_{k=0}^{\infty} (2^k L)^{ p(1 +\alpha - \beta) -2 + a_0 -a_1},
\end{align}
where
\begin{align}\label{a1}
	a_1 := \frac{p}{q}(1-\delta)(a_0 - \gamma_0) + \beta -1 
	-\frac{\gamma_1}{r} 
	+ a_0\left( 1 - \frac{p}{q}\right) .
\end{align} 

Recall from \eqref{eq9} and \eqref{eq10} that $ a_0 + p(1+\alpha - \beta) -2 <0$. Then, to obtain a finite sum in \eqref{ff}, it suffices to show that $a_1 \geq 0$. 

For $r=p+1$, it is not difficult to show by using the assumption \eqref{assp:gamma0} on $\gamma_0$ that $a_1 > a_0 \geq 0$. Now suppose that $r > p+1$, so that $q = rp/(r-1) < r$.
Firstly, assume that $a_0 = 1 -\alpha + \frac{(\gamma_0 -2)p}{r}$. From \eqref{a_0defi}, it follows that $-1 < \alpha \leq \beta -2 + \frac{2-\gamma_1}{r}$. Moreover, since $1 - \delta = (q - p)/(r-p)$, we have
\begin{align}
	a_1 &= \frac{p(q -p)}{q(r - p)} (a_0 - \gamma_0) + \beta - 1 
	- \frac{\gamma_1}{r}
	+ a_0 \frac{(q - p)}{q} \notag\\ 
	&= \frac{r(q -p)}{q(r - p)} a_0 - \frac{p(q -p)}{q(r - p)} \gamma_0 + \beta - 1 + \frac{2 - \gamma_1}{r} - 2 \frac{(q - p)}{q} \label{eq:a1} \\
	&= \frac{r(q -p)}{q(r - p)} (-1 - \alpha) + \beta - 1 + \frac{2 - \gamma_1}{r}  > 0,
	\label{a1positive}
\end{align}
where the inequality follows by using that $r > q$, which implies $r (q - p)< q (r  - p)$.

Now, if $a_0 = 3 - \beta + \frac{(\gamma_1-2)}{r} + \frac{(\gamma_0 -2)p}{r}$, then from \eqref{a_0defi} we have $\beta -2 + \frac{2-\gamma_1}{r} \leq \alpha < \beta -1 + \frac{2-\gamma_0}{r}$. Hence, from \eqref{eq:a1},
\begin{align}\label{a1positive2}
	a_1 
	&=\frac{r(q -p)}{q(r - p)} \left( 3 - \beta + \frac{(\gamma_1-2)}{r} + \frac{(\gamma_0 -2)p}{r} \right) - \frac{p(q -p)}{q(r - p)} \gamma_0 + \beta - 1 + \frac{2 - \gamma_1}{r} - 2 \frac{(q - p)}{q} \notag \\
	&=\frac{r(q -p)}{q(r - p)} \left( 3 - \beta + \frac{(\gamma_1-2)}{r} - \frac{2p}{r} \right) + \beta - 1 + \frac{2 - \gamma_1}{r} - 2 \frac{(q - p)}{q} \notag \\
	&= \left( \beta - 1 + \frac{2 - \gamma_1}{r} \right) \left( 1 - \frac{r(q - p)}{q (r - p)} \right) 
	= \left( \beta - 1 + \frac{2 - \gamma_1}{r} \right)  \frac{p (r - q)}{q(r - p)}  > 0,
\end{align}
where we used that $r> q > p$ and $\gamma_1 < r(\beta -1) + 2$.

Therefore, $a_1 > 0$, and it follows from \eqref{ff} that
\begin{align}\label{eq13}
	\int_{|y| \geq L} \frac{|\vt(y)| | \Theta(y) |^p}{|y|^{2 -\alpha - p(1+\alpha -\beta)}} dy \leq C L^{ p(1 +\alpha - \beta) -2 + a_0 -a_1}.
\end{align}

Similarly, since $a_0 +p(1+\alpha - \beta) -2 <0$ and $\alpha>-1$, we obtain for the second term in the right-hand side of \eqref{maininequality2} that 
\begin{align}\label{eq14}
	\int_{|y| \geq L} \frac{|\Theta (y)|^p}{|y|^{3 + \alpha - p(1+\alpha - \beta)}} dy 
	&\leq  \sum_{k=0}^{\infty} \frac{1}{(2^k L)^{3 +\alpha - p(1 + \alpha - \beta)}} \int_{|y| \sim 2^kL} |\Theta(y)|^p dy \notag \\
	&\leq C \sum_{k=0}^{\infty} \frac{1}{(2^k L)^{3 +\alpha - p(1 + \alpha - \beta)}} (2^kL)^{a_0} \notag \\
	&\leq C L^{p(1+ \alpha - \beta) -2 + a_0 -(1+\alpha)}.
\end{align}

Plugging  \eqref{eq13} and \eqref{eq14} into \eqref{maininequality2}, we deduce that
\begin{align}\label{b_0}
	\int_{|y| \leq L} |\Theta (y)|^p dy  &\leq  C L^{a_0 - a_1}+ cL^{a_0 -(1+\alpha)} \notag \\
	&\leq C L^{a_0 - b_0}, \quad \textrm{where }  b_0 :=\min\{a_1, 1 +\alpha\} >0.
\end{align}

Again, if $a_0 -b_0 < 0$ then the proof is finished. Otherwise, we proceed with the bootstrap argument by now leveraging \eqref{b_0} to obtain improved estimates. To put this argument into a more general form, suppose that  
\begin{align}
	\int_{|y| \leq L} |\Theta (y)|^p dy  \leq C L^\sigma \quad \mbox{with } \sigma \leq a_0. \notag 
\end{align}

From the interpolation inequality \eqref{eq11a}, we have
\begin{align}\label{eq15}
	\int_{|y| \leq L} |\Theta(y)|^{q} dy 
	\leq C L^{\sigma\delta + \gamma_0(1-\delta)},
\end{align}
where we recall that $\delta = (r - q)/(r - p)$. Then, proceeding similarly as in \eqref{ff} and recalling the definition of $a_1$ in \eqref{a1}, we obtain
\begin{align}\label{eq16}
	&\int_{|y| \geq L} \frac{|\vt (y)| | \Theta(y) |^p}{|y|^{2 -\alpha - p(1+\alpha -\beta)}} dy \notag \\
	&\qquad\leq C \sum_{k=0}^{\infty} \frac{1}{(2^k L)^{2 - \alpha - p(1 + \alpha - \beta)}} (2^k L)^{\frac{p}{q}(\sigma\delta + \gamma_0(1-\delta))} (2^k L)^{1 - \alpha -\beta + \frac{\gamma_1 }{r}} 
	\notag \\
	&\qquad\leq  C \sum_{k=0}^{\infty} (2^k L)^{p(1+\alpha -\beta) -2 +  \sigma \frac{\delta p}{q} + \gamma_0 (1-\delta) \frac{p}{q} +1 - \beta + 
	\frac{\gamma_1}{r} }
\notag \\
&\qquad\leq  C \sum_{k=0}^{\infty} (2^k L)^{p(1+\alpha -\beta) -2 + a_0 -a_1 + \sigma \frac{\delta p}{q}-a_0\frac{\delta p}{q}} \notag \\
&\qquad\leq C L^{p(1+\alpha -\beta) -2 + a_0 -a_1 + \frac{\delta p}{q} \left( \sigma - a_0 \right) },
\end{align}
where the last inequality follows from the fact that $a_0 + p(1+\alpha -\beta) -2 <0$, $a_1 >0,$ and $\sigma \leq a_0$. 

Next, similarly as in \eqref{eq14}, we obtain for the second term in the right-hand side of \eqref{maininequality2} that
\begin{align}\label{eq17}
\int_{|y| \geq L} \frac{|\Theta (y)|^p}{|y|^{3 + \alpha - p(1+\alpha - \beta)}} dy 
&\leq C \sum_{k=0}^{\infty} \frac{1}{(2^k L)^{3 +\alpha - p(1 + \alpha - \beta)}} (2^k L)^\sigma \notag \\
&\leq C L^{p(1 +\alpha -\beta) -2 + \sigma -(1 +\alpha)},
\end{align}
where the last inequality is justified by the fact that $ p(1+\alpha -\beta) -2 + \sigma \leq p(1+\alpha -\beta) -2 + a_0<0$, and $\alpha > -1$. Therefore, combining \eqref{eq16} and \eqref{eq17} with \eqref{maininequality2}, yields
\begin{align}\label{profileestimates2}
\int_{|y| \leq L} |\Theta (y)|^p dy  
\leq  C L^{a_0 -a_1 + \left( \sigma - a_0 \right)\dep} + C L^{\sigma -(1 + \alpha)}, 
\end{align}
where
\begin{align*}
\dep := \frac{\delta p}{q} \in [0,1). 
\end{align*}
Note that
\begin{empheq}[left=\displaystyle{ \int_{|y| \leq L} |\Theta(y)|^p dy \leq} \empheqlbrace]{alignat=2}
& C L^{\sigma - (1 + \alpha)} & \quad \mbox{ if }\,\, a_1 - (1 + \alpha) \geq (a_0 - \sigma) (1 - \dep), \label{ineq:sigma:1} \\
& C L^{a_0 - a_1 + (\sigma - a_0) \dep} & \quad \mbox{ if } \,\,a_1 - (1 + \alpha) < (a_0 - \sigma) (1 - \dep). \label{ineq:sigma:2}
\end{empheq}

Let us now specialize this estimate to the case $\sigma = a_0 - b_0$, as in \eqref{b_0}, where we recall that $b_0 = \min\{a_1, 1 + \alpha\}$. Firstly, suppose $b_0 =a_1$, so that $a_1 \leq 1 + \alpha$. Since $\dep < 1$, it follows from \eqref{ineq:sigma:2} with $\sigma = a_0 - a_1$ that 
\begin{align}\label{eq19}
\int_{|y| \leq L} |\Theta (y)|^p dy  \leq  C L^{a_0 - a_1\left(1 +  \dep\right)}.
\end{align}
If $a_0 - a_1(1 +\dep) < 0$, then $\Theta \equiv 0$ in $\R^2$. Otherwise, i.e. if $a_0 - a_1(1 +\dep)  \geq0$, we invoke \eqref{ineq:sigma:2} with $\sigma = a_0 - a_1(1+\dep)$ and obtain
\begin{align}\label{profileestimates3}
\int_{|y| \leq L} |\Theta (y)|^p dy  
\leq C L^{a_0 - a_1\left(1 +  \dep + \dep^2 \right)}.
\end{align}
Hence, repeating this process $n$ times, for any given $n \in \N$, we arrive at
\begin{align}\label{20}
\int_{|y| \leq L} |\Theta (y)|^p dy  
&\leq  C L^{a_0 - a_1(1 +  \dep + \dep^2 + \ldots + \dep^n)} 
= C L^{a_0 - a_1\left(\frac{1 - \delta_p^{n+1}}{1- \delta_p}\right)}.
\end{align}

Observe that 
\begin{align}\label{lim:dep}
	\frac{1 - \dep^{n+1}}{1-\dep} \rightarrow \frac{1}{1-\dep} =   \frac{q (r - p)}{r (q - p)} \quad \mbox{as } n \rightarrow \infty.
\end{align}
Moreover, recalling the definition of $a_1$ in \eqref{a1}, and particularly \eqref{eq:a1}, we have
\begin{align}\label{a_0 -a_1}
&a_0 - a_1 \frac{q (r - p)}{r (q - p)} \notag \\
&\qquad= a_0 -  \frac{q (r - p)}{r (q - p)} \left[  \frac{r(q -p)}{q(r - p)} a_0 - \frac{p(q -p)}{q(r - p)} \gamma_0 + \beta - 1 
- \frac{\gamma_1}{r}
\right] \notag \\
&\qquad=  \frac{p}{r}\left[\gamma_0  -  \frac{q (r - p)}{p (q - p)} \left( \beta  -1  
- \frac{\gamma_1}{r}
\right)\right] < 0,
\end{align} 
where the inequality follows from \eqref{cond:gamma0:q}. In view of \eqref{lim:dep} and \eqref{a_0 -a_1}, it follows that there exists $n$ sufficiently large for which the power of $L$ in \eqref{20} is negative. This implies that $\Theta \equiv 0$ in $\R^2$.

Next, let us consider the case when $b_0 = 1 + \alpha$, so that $a_1 \geq 1 + \alpha$. We apply \eqref{ineq:sigma:1}-\eqref{ineq:sigma:2} with $\sigma = a_0 - (1 + \alpha)$ and obtain
\begin{empheq}[left=\displaystyle{ \int_{|y| \leq L} |\Theta(y)|^p dy \leq} \empheqlbrace]{alignat=2}
& C L^{a_0 - 2(1 + \alpha)} & \quad \mbox{ if }\,\, a_1 - (1 + \alpha) \geq (1+ \alpha) (1 - \dep), \label{ineq:m0:1} \\
& C L^{a_0 - a_1 - (1 + \alpha) \dep} & \quad \mbox{ if } \,\,a_1 - (1 + \alpha) < (1 + \alpha) (1 - \dep). \label{ineq:m0:2}
\end{empheq}
If the powers of $L$ in both \eqref{ineq:m0:1} and \eqref{ineq:m0:2} are negative, then we conclude the proof. Otherwise, we proceed to improve on the upper bound of $\int_{|y| \leq L} |\Theta(y)|^p dy$ again via bootstrapping. To this end, we start by taking $m_0 \in \{1,2,\ldots\}$ as the smallest integer such that
\begin{align}\label{ineq:m0}
a_1 - (1 + \alpha) < m_0 (1 + \alpha) (1 - \dep).
\end{align}
If $m_0 = 1$, then \eqref{ineq:m0:2} holds. On the other hand, if $m_0 \geq 2$ then
\begin{align}\label{ineq:m0:b}
(m_0 -1)(1 + \alpha) (1 - \dep) \leq a_1 - (1 + \alpha).
\end{align}
and \eqref{ineq:m0:1} holds. In the latter case, we may repeat this computation $(m_0-1)$-times, where at each $k$th time with $k=1,\ldots,m_0-2$, we invoke \eqref{ineq:sigma:1} with $\sigma = a_0 - (k + 1)(1 + \alpha)$, and at $k = m_0-1$ we invoke \eqref{ineq:sigma:2} with $\sigma = a_0 - m_0(1 + \alpha)$. We then arrive at
\begin{align}\label{Theta:Lp:a0:b1}
\int_{|y| \leq L} |\Theta (y)|^p dy  \leq  C L^{a_0 - b_1}, \quad \mbox{where } b_1 := a_1  + m_0 (1 + \alpha) \dep.
\end{align}
Note that $b_1 \geq a_1$, and that \eqref{Theta:Lp:a0:b1} in fact holds for all $m_0 \geq 1$.

If $a_0 - b_1 < 0$, the proof is finished. Otherwise, we proceed similarly as before and apply \eqref{ineq:sigma:1}-\eqref{ineq:sigma:2} with $\sigma = a_0 - b_1$, which yields
\begin{empheq}[left=\displaystyle{ \int_{|y| \leq L} |\Theta(y)|^p dy \leq} \empheqlbrace]{alignat=2}
& C L^{a_0 - b_1 - (1 + \alpha)} & \quad \mbox{ if }\,\, a_1 - (1 + \alpha) - b_1 (1 - \dep) \geq 0,  \\
& C L^{a_0 - a_1 - b_1 \dep} & \quad \mbox{ if } \,\,a_1 - (1 + \alpha) - b_1 (1 - \dep) < 0.
\end{empheq}
Then, if necessary, we proceed by taking $m_1 \in \{0,1,2,\ldots\}$ the smallest integer such that
\begin{align*}
a_1 - (1 + \alpha) - b_1 (1 - \dep) < m_1 (1 + \alpha) (1 - \dep).
\end{align*}
After repeating this process $m_1$ times, we obtain
\begin{align*}
\int_{|y| \leq L} |\Theta (y)|^p dy \leq C L^{a_0 - b_2}, \quad \mbox{where } b_2 := a_1  + (b_1 + m_1 (1 + \alpha) ) \dep.
\end{align*}
Here we note that $b_2 \geq a_1 + a_1 \dep$.

We may keep on iteratively repeating the same argument if necessary and denote by $m_n \in \{0,1,2,\ldots\}$, for each $n \in \N$, $n=2,3,\ldots$, the smallest integer such that 
\begin{align*}
a_1 - (1 + \alpha) - b_n (1 - \dep) < m_n (1 + \alpha) (1 - \dep)
\end{align*}
to obtain
\begin{align}\label{ineq:bn}
\int_{|y| \leq L} |\Theta (y)|^p dy \leq  C L^{a_0 - b_{n+1}}, \quad \mbox{where } b_{n+1} := a_1  + (b_n + m_n (1 + \alpha) ) \dep.
\end{align}
Moreover, we have
\begin{align*}
b_{n+1} \geq a_1 (1 + \dep + \ldots + \dep^n).
\end{align*}
Therefore, by the same argument from \eqref{20}-\eqref{a_0 -a_1}, we deduce that there exists $n$ sufficiently large for which the power of $L$ in \eqref{ineq:bn} is negative, and consequently $\Theta \equiv 0$ in $\R^2$.
This concludes the proof of this case.	
\end{case}

\begin{case}\label{case3} 
Finally, suppose that $\beta -1 + \frac{2-\gamma_0}{r} \leq \alpha \leq \beta -1 + \frac{2}{p}$. In this case, we prove that either $\Theta \equiv 0$ in $\R^2$, or $\Theta \not\equiv 0$ and \eqref{growth-profile} holds. 

Assume $\Theta \not\equiv 0$. From the proof of \cref{case1}, and particularly \eqref{eq1step2}, it follows that
\begin{align}\label{Lp:upper:bound}
\int_{|y| \leq L} |\Theta(y)|^p dy \lesssim L^{2 - p(1 + \alpha - \beta)} \quad \mbox{ for all } L \gg 1.
\end{align}
Therefore, it only remains to show that 
\begin{align}\label{Lp:lower:bound}
\int_{|y| \leq L} |\Theta(y)|^p dy \gtrsim  L^{2 - p(1 + \alpha - \beta)} \quad \mbox{ for all } L \gg 1.
\end{align}

Suppose by contradiction that \eqref{Lp:lower:bound} does not hold. Then, there exists a sequence of positive numbers $L_i$, $i \in \N$, such that $L_i \to \infty$ as $i \to \infty$ and
\begin{align*}
\frac{1}{L_i^{2 - p(1 + \alpha - \beta)}} \int_{|y| \leq L_i} |\Theta(y)|^p dy \rightarrow 0 \quad \mbox{as } i \to \infty.
\end{align*}

Taking $l_2 = 4 L_i/ \rho$ and $L := \rho l_1/8 \gg 1$ in \eqref{maininequality}, it follows after taking $i \to \infty$ that
\begin{align}\label{maininequalitystep3}
\frac{1}{L^{2- p(1 + \alpha - \beta)}} \int_{|y| \leq L} |\Theta (y)|^p dy  
\leq C \int_{|y| \geq L} \frac{|\vt (y)| | \Theta(y) |^p}{|y|^{2 -\alpha - p(1+\alpha -\beta)}} dy  + C \int_{|y| \geq L} \frac{|\Theta(y)|^p}{|y|^{3 + \alpha -p(1+\alpha -\beta)}} dy.
\end{align}

We now proceed similarly as in \eqref{eq11a}-\eqref{b_0}. Namely, recalling the notation $q:= rp/(r-1)$, we obtain
by interpolation, \eqref{Lp:upper:bound} and assumption \eqref{profile-hipothesis-function} that
\begin{align*}
\int_{|y| \leq L} |\Theta(y)|^{q} dy \leq \left( \int_{|y| \leq L} |\Theta (y)|^p dy \right)^{\delta} \left( \int_{|y| \leq L} |\Theta(y)|^r dy \right)^{1 -\delta} 
\leq C L^{(2 -p(1+\alpha -\beta))\delta + \gamma_0 (1 - \delta)},
\end{align*}
where $\delta = (r - q)/(r - p)$.

Proceeding analogously as in \eqref{ff} and recalling that $\dep := \delta p / q$, we estimate
\begin{align}\label{tilde_a0_negativo}
&\int_{|y| \geq L} \frac{|\vt (y)| | \Theta(y) |^p}{|y|^{2 -\alpha - p(1+\alpha -\beta)}} dy \notag \\
&\quad\leq \sum_{k=0}^{\infty} \frac{1}{(2^k L)^{2 - \alpha - p(1 + \alpha - \beta)}} \left( \int_{|y| \sim 2^k L} |\Theta (y)|^q dy \right)^{\frac{p}{q}} \left( \int_{|y| \sim 2^kL} |\vt (y)|^{r} dy \right)^{\frac{1}{r}} 
\notag \\
&\quad\leq C \sum_{k=0}^{\infty} (2^k L)^{p(1 + \alpha - \beta) +\alpha -2} (2^kL)^{\frac{p}{q}\left(2-p(1+\alpha -\beta)\right)\delta + \gamma_0 \frac{p}{q}(1-\delta)} (2^kL)^{1 - \alpha -\beta + \frac{\gamma_1}{r}} 
\notag \\
&\quad\leq C L^{(p(1+\alpha - \beta) -2)\left(1 - \dep \right) +  \gamma_0 \frac{p}{q}(1-\delta)  + 1-\beta +
	\frac{\gamma_1}{r}},
\end{align}
where in the last inequality we used that $\alpha \leq \beta - 1 + \frac{2}{p}$ and condition \eqref{cond:gamma0:q}.

Moreover, analogously as in \eqref{eq14}, we obtain
\begin{align}\label{eq5step3}
\int_{|y| \geq L} \frac{|\Theta (y)|^p}{|y|^{3 + \alpha - p(1+\alpha - \beta)}} dy 
&\leq \sum_{k=0}^{\infty} \frac{1}{(2^k L)^{3 +\alpha - p(1 + \alpha - \beta)}} \int_{|y| \sim 2^kL} |\Theta(y)|^p dy \notag \\
&\leq C \sum_{k=0}^{\infty} (2^k L)^{p(1 + \alpha - \beta) -3 -\alpha} (2^k L)^{2 -p(1+\alpha -\beta)} \leq C L^{-(1+\alpha)},
\end{align}
where we recall that $\alpha > -1$.

Thus, from \eqref{maininequalitystep3}, 
\begin{align*}
\frac{1}{L^{2- p(1 + \alpha - \beta)}} \int_{|y| \leq L} |\Theta (y)|^p dy  
\leq C L^{(p(1+\alpha - \beta) -2)\left(1 - \dep \right) +  \gamma_0  \frac{p}{q}(1-\delta)  + 1-\beta +
\frac{\gamma_1}{r}}
+ C L^{-(1+\alpha)}.
\end{align*}
Note that the power of $L$ in the first term from the right-hand side cannot be smaller than the power of $L$ in the second term. Indeed, since $1 - \dep= \frac{r(q - p)}{q(r - p)}$, $\alpha \geq \beta -1 + \frac{2-\gamma_0}{r}$, $1/r = 1 - p/q$, and $\gamma_0 \leq \gamma_1 + r$, we deduce that
\begin{align}\label{ineq:tildea0}
&(p(1+\alpha - \beta) -2)\left(1 - \dep\right) + \gamma_0 \frac{p}{q}(1-\delta)  + 1-\beta + 
\frac{\gamma_1}{r}
+ 1 + \alpha \notag \\
&\;\;\geq  \frac{p(q - p)}{q(r - p)}(2 - \gamma_0) -2 \frac{r(q-p)}{q(r-p)} + \gamma_0 \frac{p(q - p)}{q(r - p)} + 2 -\beta + 
\frac{\gamma_1}{r}
+ \beta -1 + \frac{2-\gamma_0}{r} \notag\\
&\;\;= \frac{2 (p-r) (q - p)}{q(r - p)} + 1 + \frac{\gamma_1}{r} + 2 \left( 1 - \frac{p}{q}\right) - \frac{\gamma_0}{r}
= 1 +  \frac{\gamma_1}{r}  - \frac{\gamma_0}{r} \geq 0.
\end{align}

Hence,
\begin{align}\label{eq8step3}
\int_{|y| \leq L} |\Theta (y)|^p dy \leq C L^{2-p(1+\alpha -\beta) - d_0},
\end{align}
where
\begin{align}\label{eqa0step3}
d_0 := (2 -p(1+\alpha - \beta))\left(1 - \dep\right) -  \gamma_0 \frac{p}{q}(1-\delta)  - 1 + \beta  
 - \frac{\gamma_1}{r}
> 0.
\end{align}
If $2-p(1+\alpha -\beta) - d_0 < 0$ then \eqref{eq8step3} implies that $\Theta \equiv 0$ in $\R^2$, which yields a contradiction and finishes the proof. Otherwise, for $2 - p (1+\alpha -\beta) - d_0 \geq 0$, we repeat the above argument by using now the improved estimate \eqref{eq8step3}. This gives
\begin{align}\label{eqstep3a}
\int_{|y| \leq L} |\Theta(y)|^{q} dy 
\leq C L^{(2 -p(1+\alpha -\beta))\delta - d_0 \delta + (1-\delta)\gamma_0}
\end{align}
and
\begin{align*}
&\frac{1}{L^{2- p(1 + \alpha - \beta)}} \int_{|y| \leq L} |\Theta (y)|^p dy  \\
&\qquad\leq  C \sum_{k=0}^{\infty} (2^k L)^{p(1 + \alpha - \beta) +\alpha -2} (2^k L)^{ \dep\left(2-p(1+\alpha -\beta)\right) - d_0\dep + \gamma_0 \frac{p}{q}(1-\delta) + 1-\alpha - \beta + \frac{\gamma_1}{r} 
} \\
&\qquad\qquad+ C \sum_{k=0}^{\infty} (2^k L)^{p(1 + \alpha - \beta) -3 -\alpha} (2^k L)^{2 -p(1+\alpha -\beta) - d_0 } \\
&\qquad\leq C \sum_{k=0}^{\infty} (2^k L)^{(p(1+\alpha -\beta) -2)\left(1 - \dep \right) + \gamma_0 \frac{p}{q}(1-\delta) + 1-\beta +
\frac{\gamma_1}{r}
-d_0 \dep}
+ C \sum_{k=0}^{\infty} (2^k L)^{-(1+\alpha) -d_0} \\
&\qquad\leq C L^{- d_0(1 + \dep)} + C L^{-(1+\alpha) -d_0} 
\leq C L^{- d_0(1 + \dep)},
\end{align*}
where we used that $0 < d_0 \leq 1 + \alpha$, according to \eqref{ineq:tildea0}, \eqref{eqa0step3}. Thus,
\begin{align*}
\int_{|y| \leq L} |\Theta (y)|^p dy 
\leq C L^{2- p(1 + \alpha - \beta) - d_0(1 + \dep)}.
\end{align*} 

We may repeat this process for as many $n$ times, $n \in \N$, as necessary, to obtain that
\begin{align}\label{int:Theta:p:3}
\int_{|y| \leq L} |\Theta (y)|^p dy \leq  C L^{2 - p(1+\alpha -\beta) - d_0 (1 +  \dep + \dep^2 + \ldots + \dep^n)}
= C L^{2 - p(1+\alpha -\beta) - d_0  \left( \frac{1 - \dep^{n+1}}{1 - \dep}\right)}.
\end{align}
As before, note that  $ \frac{1 - \dep^{n+1}}{1 - \dep} \to \frac{1}{1 - \dep} = \frac{q(r - p)}{r(q - p)}$ as $n \to \infty$. Additionally, from the definition of $d_0$ in \eqref{eqa0step3} and condition \eqref{cond:gamma0:q}, we have
\begin{align}\label{eq12step3}
&2-p(1+\alpha -\beta) - d_0 \frac{q(r - p)}{r(q - p)} \notag \\
&\qquad=2-p(1+\alpha -\beta) - \frac{q(r - p)}{r(q - p)} \left[ (2 -p(1+\alpha - \beta)) \frac{r(q-p)}{q(r-p)} - \gamma_0 \frac{p(q - p)}{q(r-p)}  - 1 + \beta 
- \frac{\gamma_1}{r}  \right] 
\notag \\
&\qquad= \frac{p}{r}\left[\gamma_0  - \frac{q(r - p)}{p(q - p)} \left( \beta  -1  
- \frac{\gamma_1}{r}
\right)\right] < 0.
\end{align}
Therefore, we may take $n$ sufficiently large for which the power of $L$ in \eqref{int:Theta:p:3} is negative. This implies
that $\Theta \equiv 0$ in $\R^2$, which is a contradiction with our starting assumption that $\Theta \not\equiv 0$ in $\R^2$. This concludes the proof.
\end{case}

\subsection{Proof of \cref{thm:main3}}

We start with the proof of \ref{cor:1:i}. Let $M$ be a positive constant such that 	$|\Theta (y)| \lesssim |y|^{-\sigma_0}$ and $|\nabla \Theta (y)| \lesssim |y|^{-\sigma_1}$ for all $|y| \geq M$. Thus, since $\Theta \in \cC^1(\R^2)$, it follows that for all $L > 0$ and $r > \max \left\{ \frac{2}{\sigma_0}, \frac{2}{\sigma_1} \right\}$, we have
\begin{align*}
\int_{|y| \leq L} |\Theta (y)|^r dy \leq \int_{|y| \leq M} |\Theta (y)|^r dy + \int_{|y| \geq M} \frac{1}{|y|^{r\sigma_0}} dy \leq C,
\end{align*}
and
\begin{align*}
\int_{|y| \leq L} |\nabla \Theta (y)|^r dy \leq \int_{|y| \leq M} |\nabla \Theta (y)|^r dy + \int_{|y| \geq M} \frac{1}{|y|^{r\sigma_1}} dy \leq C.
\end{align*}

Let $p_1 :=  \max \left\{1, \frac{2}{\sigma_0}, \frac{2}{\sigma_1}\right\}$. Then, the assumptions of \cref{thm:main} are satisfied with the following parameter choices: $p = p_1$, $r = p_1 +1$, $\gamma_0 = 0$, and $ \gamma_1 =0$. Consequently, the values of $\alpha$ that admit a nontrivial corresponding profile $\Theta$ belong to the interval $\left[\beta -1 + \frac{2}{p_1+1}, \beta -1 + \frac{2}{p_1} \right]$. On the other hand, note that the assumptions of \cref{thm:main} are also satisfied with $p = p_1 + k$, $r = p_1 + k+1$, $\gamma_0 = 0$, and $ \gamma_1 =0$, for any $k > 0$. This implies that the values of $\alpha$ admitting nontrivial profiles must also belong to the interval $\left[\beta -1 + \frac{2}{p_1+k + 1}, \beta -1 + \frac{2}{p_1 + k} \right]$ for any $k >0$. In particular, for $k \geq 2$, we obtain that $\alpha \in \left[\beta -1 + \frac{2}{p_1 +k+1}, \beta -1 + \frac{2}{p_1 + k} \right] \cap \left[\beta -1 + \frac{2}{p_1 + 1}, \beta -1 + \frac{2}{p_1}\right] = \emptyset$. Therefore, we conclude that $\Theta \equiv 0$ in $\R^2$.

We proceed to prove \ref{cor:1:ii}. Let $M > 0$ such that $ |\nabla \Theta (y)| \lesssim |y|^{\sigma_1}$ for all $|y| \geq M$, and fix any $p \in [1,\infty)$. Choose  $r\geq p+1$ sufficiently large such that $r > (2 + \beta p)/(\beta - 1 - \sigma_1)$.
Observe that, for all $L \gg 1$,
\begin{align}\label{int:DTheta:Lr}
\int_{|y| \leq L} |\nabla \Theta (y)|^r dy &\leq \int_{|y| \leq M} |\nabla \Theta (y)|^r dy + \int_{M \leq |y| \leq L} |y|^{\sigma_1 r} dy \notag \\
&\leq C \int_{|y| \leq M} dy + L^{\sigma_1 r} \int_{M \leq |y| \leq L} dy
\leq C L^{\sigma_1r + 2},
\end{align}
where we used that $\Theta \in \cC^1(\R^2)$. Then, it follows by Sobolev embedding that
\begin{align*}
\int_{|y| \leq L} |\Theta (y)|^r dy \leq C L^{(\sigma_1 + 1) r +2}.
\end{align*}

Therefore, the assumptions of \cref{thm:main} are satisfied by setting $\gamma_1 =\sigma_1r + 2$ and $\gamma_0 = (\sigma_1 + 1)r + 2 = \gamma_1 + r$. It follows that the values of $\alpha$ admitting nontrivial profiles belong to the interval $\left[\beta - 2 -\sigma_1, \beta -1 + \frac{2}{p} \right]$ and the corresponding profile satisfies
\begin{align}\label{eq5teo2}
C_1 L^{2-p(1+\alpha -\beta)} \leq \int_{|y| \leq L}  |\Theta(y)|^p dy \leq  C_2 L^{2-p(1+\alpha -\beta)} \quad \mbox{for all } L \gg 1, 
\end{align}
for some positive constants $C_1,C_2$. On the other hand, since $|\Theta (y)| \gtrsim 1$ for $|y| \gg 1$, it follows that
\begin{align}\label{eq6teo2}
\int_{|y| \leq L} |\Theta(y)|^p dy  \geq CL^{2} \quad \mbox{ for } L \gg 1.
\end{align}
Combining the upper bound in \eqref{eq5teo2} with \eqref{eq6teo2}, we must have $1 + \alpha -\beta \leq 0$, which implies that the values of $\alpha$ admitting nontrivial profiles in fact belong to $[\beta -2 - \sigma_1 ,\beta -1]$. This completes the proof.

\subsection{Proof of \cref{thm:main2}}

The proof of \cref{thm:main2} follows similar ideas from the proof of \cref{thm:main}. Thus, we will keep the presentation here shorter and refer to the analogous calculations above as needed, while expanding on the main differences. As in \cref{thm:main}, the proof is divided into three different cases, each corresponding to a particular range for $\alpha$ within the interval $(-1, \infty)$. Without loss of generality, we assume again that $x_0 = 0$.

The first case is when $\alpha > \beta+\frac{2}{p} -1$, and the proof is identical to \cref{case1} above. The second case is when $-1 < \alpha < \beta -1 + \frac{2-\gamma}{r}$, and as in \cref{case2} we will start from the local $L^p$ equality \eqref{eq2} and estimate each of its terms. The difference lies at the estimate of $\bu$ due to better integrability properties of the kernel $K_\beta$ when $\beta \in (0,1]$, in comparison to when $\beta \in (1,2)$. Namely, as pointed out in \cref{sec:intro}, $K_\beta$ is integrable near the origin for $0< \beta <1$, and is a Calder\'on-Zygmund operator for $\beta = 1$. Then, taking as before a cut-off function $\phi_\rho \in \cC^\infty(\R^2)$ with $0 \leq \phi_\rho \leq 1$, $\phi_\rho \equiv 1$ in $B_{\rho/2}(0)$, and $\phi_\rho \equiv 0$ in $B_\rho^c(0)$, we proceed similarly as in \cite[Section 2]{Xue16} by decomposing $\bu$ into a term involving the self-similarity region and another one outside of it:
\begin{align}\label{eq0-thm2}
\bu(x,t)
&= C_\beta P.V. \int_{\R^2} K_\beta (x-y) \theta(y,t)\phi_\rho (y) dy +   C_\beta P.V. \int_{\R^2} K_\beta (x-y) \theta(y,t)(1 - \phi_\rho(y)) dy \notag \\
&=: {\tbu}^{(1)}(x,t) + {\tbu}^{(2)}(x,t). 
\end{align}

By the local self-similarity of $\theta$, analogously to \eqref{eq4}, we can rewrite ${\tbu}^{(1)}$ as
\begin{align*}
{\tbu}^{(1)}(x,t) = \frac{C_\beta}{(T-t)^{\frac{\alpha}{1 + \alpha}}} U^{(1)}\bigg( \frac{x}{(T-t)^{\frac{1}{1 + \alpha}}}, t\bigg),
\end{align*}
where
\begin{align*}
U^{(1)}(x,t) := P.V. \int_{\R^2}  K_\beta(x-y) \Theta(y) \phi_{\rho}(y(T-t)^{\frac{1}{1+\alpha}}) dy. 
\end{align*}
Note that  ${\tbu}^{(2)}$ has the same expression as $\bu^{(3)}$ defined in \eqref{eq3}. Since $K_\beta$ is square integrable in any region that does not contain the origin,  for any $\beta>0$, then, for $\rho/8\leq |x|\leq \rho/4$, we obtain as in \eqref{eq6} that
\begin{align}\label{u2-beta01}
|\tbu^{(2)}(x,t)| \leq  C_\beta \|\theta(0)\|_{L^2}.
\end{align}
Proceeding with an analogous computation as in \eqref{bb}-\eqref{eq7}, it follows that the right-hand side of \eqref{eq2} can be estimated by
\begin{align}\label{eq7_b01}
&\bigg| \int_{t_1}^{t_2}  \int_{\R^2}  |\theta(x,t)|^p  (\bu(x,t) \cdot \nabla \phi_{\frac{\rho}{4}}(x))  dx \, dt \bigg| \notag \\
&\qquad\leq C \int_{\frac{\rho}{8} l_1 \leq |y| \leq \frac{\rho}{4}l_2} \frac{|\tU (y)| | \Theta(y) |^p}{|y|^{2 -\alpha - p(1+\alpha -\beta)}} dy  + C \int_{\frac{\rho}{8} l_1 \leq |y| \leq \frac{\rho}{4}l_2} \frac{|\Theta(y)|^p}{|y|^{3 + \alpha -p(1+\alpha -\beta)}} dy ,
\end{align}
where
\begin{equation}
\tU(y) =  \int_{t_1}^{t_2}  \left|P.V. \int_{\R^2} K_\beta(y-z) \Theta(z) \phi_{\rho}(z(T-t)^{\frac{1}{1+\alpha}}) dz \right|  \ind_{\Ay}(t)dt,\notag
\end{equation}
and $\ind_{\Ay}$ is the indicator function of the set $\Ay$ defined in \eqref{ball-defi}. Plugging this back into \eqref{eq2} and recalling that
\begin{align}
\int_{\R^2} |\theta(x, t_i)|^p \phi_{\frac{\rho}{4}} (x) dx 
&= l_i^{p(1 + \alpha - \beta) -2} \int_{|y| \leq \frac{\rho}{4}l_i} |\Theta (y)|^p \phi_{\frac{\rho}{4}} (y l_{i}^{-1}) dy, \notag
\end{align}
where $l_i = (T-t_i)^{- \frac{1}{1+\alpha}}$, $i =1, 2$, we obtain that 
\begin{multline}\label{maininequality-thm2-b01} 
\bigg| l_2^{p(1 + \alpha - \beta) -2} \int_{\R^2} |\Theta (y)|^p \phi_{\frac{\rho}{4}} (y l_{2}^{-1}) dy - l_1^{p(1 + \alpha - \beta) -2} \int_{\R^2} |\Theta (y)|^p \phi_{\frac{\rho}{4}} (y l_{1}^{-1}) dy \bigg| \\
\leq C \int_{\frac{\rho}{8} l_1 \leq |y| \leq \frac{\rho}{4}l_2} \frac{|\tU(y)| | \Theta(y) |^p}{|y|^{2 -\alpha - p(1+\alpha -\beta)}} dy  + C \int_{\frac{\rho}{8} l_1 \leq |y| \leq \frac{\rho}{4}l_2} \frac{|\Theta(y)|^p}{|y|^{3 + \alpha -p(1+\alpha -\beta)}} dy.
\end{multline}

Reproducing the same steps as in \eqref{limitL2}-\eqref{maininequality2}, we obtain 
\begin{align}\label{maininequality2-thm2}
\frac{1}{L^{2- p(1 + \alpha - \beta)}} \int_{|y| \leq L} |\Theta (y)|^p dy  
\leq C \int_{|y| \geq L} \frac{|\tU (y)| | \Theta(y) |^p}{|y|^{2 -\alpha - p(1+\alpha -\beta)}} dy  + C \int_{|y| \geq L} \frac{|\Theta(y)|^p}{|y|^{3 + \alpha -p(1+\alpha -\beta)}} dy.
\end{align}
Next, proceeding analogously to \eqref{eq9}-\eqref{eq10} and invoking the upper estimate on the $L^r$ norm of $\tU$ on a dyadic shell provided by \cref{lemmamain2}, we get that for $L$ sufficiently large

\begin{align}\label{eqa0}
\int_{|y| \leq L} |\Theta (y)|^p dy  &\leq  C L^{\frac{(p+1)(\gamma-2)}{r} -\beta +2} + C L^{1 -\alpha + \frac{(\gamma -2)p}{r}} \leq C L^{\ta_0},
\end{align}
where
\begin{align}\label{def:a0:thm2}
\ta_0:= 1 - \alpha + \frac{p(\gamma -2)}{r}.
\end{align}
Clearly, if $\ta_0 <0$ then $\Theta \equiv 0$ in $\R^2$ and the proof is finished. Otherwise, if $\ta_0 \geq 0$, we use the new upper bound \eqref{eqa0} to improve the estimates of the terms on the right-hand side of \eqref{maininequality2-thm2}. 

By interpolation as in \eqref{eq11} with $q = p+1$, we have
\begin{align}\label{interpolation-thm2}
\int_{|y| \leq L} |\Theta(y)|^{p+1} dy 
&\leq C L^{\ta_0\delta + (1-\delta)\gamma}, \quad  \textrm{with }  \delta := \frac{r-p-1}{r-p} \in [0,1). 
\end{align}

Next, by applying the dyadic decomposition once again and also H\"older's inequality, we obtain from \eqref{def:a0:thm2}, \eqref{interpolation-thm2}, and \cref{lemmamain2} that the first term in the right-hand side of \eqref{maininequality2-thm2} can be estimated as
\begin{align}\label{eq3-thm2}
&\int_{|y| \geq L} \frac{|\tU (y)| | \Theta(y) |^p}{|y|^{2 -\alpha - p(1+\alpha -\beta)}} dy \notag \\
&\leq  \sum_{k=0}^{\infty} \frac{1}{(2^k L)^{2 - \alpha - p(1 + \alpha - \beta)}} \left( \int_{|y| \sim 2^kL} |\Theta (y)|^{p+1} dy \right)^{\frac{p}{p+1}} \left( \int_{|y| \sim 2^kL} |\tU (y)|^{p+1} dy \right)^{\frac{1}{p+1}}
\notag \\
& \leq  \sum_{k=0}^{\infty} \frac{1}{(2^k L)^{2 - \alpha - p(1 + \alpha - \beta)}} (2^k L)^{\frac{p}{p+1}(\Tilde{a}_0 \delta + \gamma(1-\delta))} (2^k L)^{\frac{1}{p+1}(\Tilde{a}_0 \delta + \gamma(1-\delta) - (p+1)(\alpha + \beta))}  
\notag \\
&\leq \sum_{k=0}^{\infty} (2^k L)^{p(1+\alpha - \beta) -2 + \alpha + \tilde{a}_0 \delta + \gamma (1-\delta)  - \alpha - \beta} 
\notag \\
&\leq \sum_{k=0}^{\infty} (2^k L)^{p(1 +\alpha -\beta) -2 + \tilde{a}_0
+ (\gamma - \tilde{a}_0)(1 - \delta) - \beta }	
\notag \\
&\leq \sum_{k=0}^{\infty} (2^k L)^{p(1+\alpha -\beta) -2 +\Tilde{a}_0 - \Tilde{a}_1},
\end{align}
where
\begin{align}\label{a1-thm2}
\tilde{a}_1:=  (\tilde{a}_0 - \gamma)(1 - \delta) + \beta.
\end{align}
Here we note a relevant difference between the estimate in \eqref{eq3-thm2} and the analogous one done for the case $1<\beta<2$ in \cref{thm:main}, namely \eqref{ff}. Specifically, since \cref{lemmamain2} only requires an assumption on $\Theta$, when invoking H\"older's inequality we can apply the $L^q$ norm with $q = p+1$ on $\tU$ and thus also make use of the new bound \eqref{eqa0} in estimating this factor via \eqref{interpolation-thm2} and \eqref{firstlemmahypothesis}-\eqref{lemma2-betamenor1}, in addition to the first factor involving the $L^q = L^{p+1}$ norm of $\Theta$. By contrast, in \eqref{ff} we apply the $L^r$ norm on $\vt$ rather than $L^q$ with $q = rp/(r-1)$ since the assumption in \cref{lemmamain} involves $\nabla \Theta$, for which we obtain no further estimate besides the one provided in assumption \eqref{profile-hipothesis-derivada}. Clearly, a similar difference in the estimates applies throughout the remaining of this bootstrapping procedure. 

Now observe that since $-1<\alpha<\beta-1+\frac{2-\gamma}{r}$ and $\tilde{a}_0= 1-\alpha+\frac{p(\gamma-2)}{r}$, then $ \tilde{a}_0 + p(1+\alpha - \beta) -2 <0$. Thus, to ensure that the sum in \eqref{eq3-thm2} is finite, it suffices to prove that $\ta_1 \geq 0$. Recalling that $1 -\delta = \frac{1}{r-p}$ and invoking once again the assumption that $\alpha < \beta -1 + \frac{2-\gamma}{r}$, we obtain
\begin{align}\label{a1-positive-thm2}
\ta_1 &= (\tilde{a}_0 - \gamma)(1 - \delta) + \beta \notag\\
&= \left( 1 - \alpha + \frac{p (\gamma - 2)}{r} - \gamma \right) \frac{1}{r - p} + \beta \notag\\
&> \left( 1 - \beta + 1 - \frac{(2-\gamma)}{r} + \frac{p (\gamma - 2)}{r} - \gamma \right) \frac{1}{r - p} + \beta \notag\\
&= \left( \frac{(p+1 - r)(\gamma - 2)}{r} - \beta \right) \frac{1}{r - p} + \beta \notag\\
&= \frac{r - p - 1}{r - p} \left( \frac{2 - \gamma}{r} + \beta \right) \geq 0,
\end{align}
where in the last inequality we used the assumptions that $r \geq p+1$ and $\gamma < \beta (r - p) < \beta r + 2$. Therefore, we conclude from \eqref{eq3-thm2} that
\begin{align}\label{eq4-thm2}
&\int_{|y| \geq L} \frac{|\tU (y)| | \Theta(y) |^p}{|y|^{2 -\alpha - p(1+\alpha -\beta)}} dy \leq CL^{p(1+\alpha -\beta) -2 +\tilde{a}_0 - \tilde{a}_1}.
\end{align}
For the second term in the right-hand side of \eqref{maininequality2-thm2}, invoking \eqref{eqa0} and recalling that  $\tilde{a}_0 +p(1+\alpha - \beta) -2 <0$ and $-1 < \alpha$, we estimate
\begin{align}\label{eq5-thm2}
\int_{|y| \geq L} \frac{|\Theta (y)|^p}{|y|^{3 + \alpha - p(1+\alpha - \beta)}} dy 
&\leq  \sum_{k=0}^{\infty} \frac{1}{(2^k L)^{3 +\alpha - p(1 + \alpha - \beta)}} \int_{|y| \sim 2^kL} |\Theta(y)|^p dy \notag \\
&\leq C \sum_{k=0}^{\infty} \frac{1}{(2^k L)^{3 +\alpha - p(1 + \alpha - \beta)}} (2^kL)^{\tilde{a}_0} \notag \\
&\leq C L^{p(1+ \alpha - \beta) -2 + \Tilde{a}_0 -(1+\alpha)}.
\end{align}
Plugging  \eqref{eq4-thm2} and \eqref{eq5-thm2} into \eqref{maininequality2-thm2}, we deduce that
\begin{align}\label{b_0-them2}
\int_{|y| \leq L} |\Theta (y)|^p dy 
&\leq C L^{\ta_0 - \tb_0}, \quad \textrm{where }  \tb_0 :=\min\{\ta_1, 1 +\alpha\} > 0,
\end{align}
with $\ta_1 > 0$ as given in \eqref{a1-thm2}.

Now, analogously to \eqref{profileestimates2}, let us obtain a more general form of an estimate of the profile to help us proceed with the bootstrapping strategy. Namely, suppose that 
\begin{align}
\int_{|y| \leq L} |\Theta (y)|^p dy  \leq C L^\sigma, \quad \mbox{with } \sigma \leq \ta_0. \notag
\end{align}
Then, proceeding similarly as in \eqref{interpolation-thm2}-\eqref{eq3-thm2} and \eqref{eq5-thm2}, cf. \eqref{eq15}-\eqref{profileestimates2},
it is not difficult to arrive at
\begin{align}\label{profileestimates2-thm2}
\int_{|y| \leq L} |\Theta (y)|^p dy  
\leq  C L^{\ta_0 -\ta_1 + \left( \sigma - \ta_0 \right) \delta} + C L^{\sigma -(1 + \alpha)}.
\end{align}
Hence,
\begin{empheq}[left=\displaystyle{ \int_{|y| \leq L} |\Theta(y)|^p dy \leq} \empheqlbrace]{alignat=2}
& C L^{\sigma - (1 + \alpha)} & \quad \mbox{ if }\,\, \ta_1 - (1 + \alpha) \geq (\ta_0 - \sigma) (1 - \delta), \label{ineq:sigma:1-thm2} \\
& C L^{\ta_0 - \ta_1 + (\sigma - \ta_0) \delta} & \quad \mbox{ if } \,\,\ta_1 - (1 + \alpha) < (\ta_0 - \sigma) (1 - \delta). \label{ineq:sigma:2-thm2}
\end{empheq}

Next, we specialize this estimate to the case $\sigma = \ta_0 - \tb_0$, in view of \eqref{b_0-them2}. Firstly, suppose $\tb_0 = \ta_1$, so that $\ta_1 \leq 1 + \alpha$. Then, invoking \eqref{profileestimates2-thm2} with $\sigma = \ta_0 - \ta_1$ yields
\begin{align}\label{b_0=1+alpha}
\int_{|y| \leq L} |\Theta (y)|^p dy  
\leq  C L^{\ta_0 - \ta_1 (1 + \delta)} + C L^{\ta_0 - \ta_1 - (1 + \alpha)}.
\end{align}
Since $\ta_1 \geq 1 + \alpha > 0$, and $\delta \in [0,1)$, then $\ta_1 \delta < \ta_1 \leq 1 + \alpha$.
It thus follows that
\begin{align*}
\int_{|y| \leq L} |\Theta (y)|^p dy  
\leq  C L^{\ta_0 - \ta_1 (1 + \delta)}.
\end{align*}
We now proceed similarly as in the case $b_0 = a_1$ in the proof of \cref{thm:main}. Namely, we may repeat this process $n$ times, for any given $n \in \N$, where at each $k$th time, $k \in \{1,\ldots, n\}$, we set $\sigma = \ta_0 - \ta_1 (1 + \delta + \ldots + \delta^k)$, and arrive at
\begin{align}\label{ineq:Theta:p:4}
\int_{|y| \leq L} |\Theta (y)|^p dy  
&\leq  C L^{\ta_0 - \ta_1 (1 + \delta + \ldots + \delta^n)}. 
\end{align}
Note that $1 + \delta + \ldots + \delta^n \rightarrow \frac{1}{1 - \delta} = r- p$ as $n \to \infty$. Moreover, recalling that $\ta_1 = (\ta_0 - \gamma)(1 - \delta) + \beta$, we have
\begin{align}\label{eq_seq_delta}
\ta_0 - \ta_1 (r - p) &= \ta_0 - (\ta_0 - \gamma)(1 - \delta) (r - p) - \beta (r -p) \notag\\
&= \gamma - \beta (r - p) < 0,
\end{align}
where we used the assumption that $\gamma < \beta (r - p)$. Hence, it follows that there exists $n$ sufficiently large such that $\ta_0 - \ta_1 (1 + \delta + \ldots + \delta^n) < 0$. For such $n$, we thus conclude from \eqref{ineq:Theta:p:4} that $\Theta \equiv 0$ in $\R^2$.

For the other case, $\tb_0 = 1 + \alpha$, where $1 + \alpha \leq \ta_1$, we proceed similarly to the case $b_0 = 1+\alpha$ in the proof of \cref{thm:main}. We invoke \eqref{profileestimates2-thm2} with $\sigma = \ta_0 - (1 + \alpha)$ to obtain
\begin{empheq}[left=\displaystyle{ \int_{|y| \leq L} |\Theta(y)|^p dy \leq} \empheqlbrace]{alignat=2}
& C L^{\ta_0 - 2 (1 + \alpha)} & \quad \mbox{ if }\,\, \ta_1 - (1 + \alpha) \geq (1+ \alpha) (1-\delta), \label{ineq:m0:1-thm2} \\
& C L^{\ta_0 - \ta_1 - (1+ \alpha)\delta} & \quad \mbox{ if } \,\,\ta_1 - (1 + \alpha) < (1 + \alpha) (1-\delta). \label{ineq:m0:2-thm2}
\end{empheq}
The proof is over if the powers of $L$ in both \eqref{ineq:m0:1-thm2} and \eqref{ineq:m0:2-thm2} are negative. Otherwise, we may repeat the same process as in \eqref{ineq:m0}-\eqref{ineq:bn} and obtain, for each $n \in \N$, a number $\tb_{n+1}$ such that
\begin{align*}
\int_{|y| \leq L} |\Theta (y)|^p dy \leq  L^{\ta_0 - \tb_{n+1}},
\end{align*}
and with
\begin{align*}
\tb_{n+1} \geq \ta_1 (1 + \delta + \ldots + \delta^n).
\end{align*}

Therefore, by proceeding similarly to \eqref{ineq:Theta:p:4} and \eqref{eq_seq_delta}, we infer that 
there exists $n \in \N$ sufficiently large such that $\ta_0 - \ta_1 (1 + \delta + \ldots + \delta^n) < 0$, which implies that $\Theta \equiv 0$ and concludes the proof of this case.

Finally, for the case $\beta -1 + \frac{2-\gamma}{r} \leq \alpha \leq \beta -1 + \frac{2}{p}$ we replicate the argument by contradiction employed in \cref{case3}. Namely, we assume that $\Theta\not\equiv 0$ and, given that the upper bound in \eqref{growth-profile-thm2} is valid, we assume that the lower bound in \eqref{growth-profile-thm2} does not hold. Then, analogously to \eqref{maininequalitystep3}, we obtain

\begin{align}\label{maininequality-thm2-case3} 
\frac{1}{L^{2- p(1 + \alpha - \beta)}} \int_{|y| \leq L} |\Theta (y)|^p dy
\leq C \int_{|y| \geq L} \frac{|\tU(y)| | \Theta(y) |^p}{|y|^{2 -\alpha - p(1+\alpha -\beta)}} dy  + C \int_{|y|\geq L } \frac{|\Theta(y)|^p}{|y|^{3 + \alpha -p(1+\alpha -\beta)}} dy.
\end{align}

This estimate will be repeatedly used in conjunction with  \eqref{interpolation-thm2} and \cref{lemmamain2} to improve the upper bound of $ \int_{|y| \leq L} |\Theta (y)|^p dy$. To present this in a more organized manner, assume that
\begin{align}
\int_{|y| \leq L} |\Theta (y)|^p dy  \leq C L^\sigma \quad \mbox{with } \sigma \leq 2-p(1+\alpha-\beta). \notag 
\end{align}
Then,  based on estimates similar to \eqref{interpolation-thm2}-\eqref{eq5-thm2}, cf. \eqref{eq15}-\eqref{eq17}, we obtain from \eqref{maininequality-thm2-case3} the following general estimate
\begin{align}\label{general-est}
\int_{|y| \leq L} |\Theta (y)|^p dy \leq C L^{D - \td_0+(\sigma-D) \delta}+CL^{\sigma -(1+\alpha)},
\end{align}
where 
\begin{align}\label{eq_D}
D:=2 -p(1+\alpha - \beta) \mbox{ and } \td_0 := (D - \gamma)(1 - \delta) + \beta
> 0.
\end{align}

In view of \eqref{eq1step2}, we first apply \eqref{general-est} with $\sigma=D$, and arrive at
\begin{align}\label{eq15-thm2}
\int_{|y| \leq L} |\Theta (y)|^p dy \leq C L^{D- \td_0}+CL^{D -(1+\alpha)}.
\end{align}

From \eqref{eq_D} and the assumption that $\alpha \geq \beta -1 + \frac{2-\gamma}{r}$, it follows that $\tilde{d}_0 \leq 1 + \alpha$. Thus,
\begin{align*}
\int_{|y| \leq L} |\Theta (y)|^p dy \leq C L^{D - \tilde{d}_0}.
\end{align*}

Next, applying \eqref{general-est} with $\sigma=D - \tilde{d}_0 $, we obtain that
\begin{align}\label{eq16a-thm2}
\int_{|y| \leq L} |\Theta (y)|^p dy \leq C L^{D- \td_0(1+ \delta)}+CL^{D - \tilde{d}_0  -(1+\alpha)}.
\end{align}
Since $\delta \in [0,1)$ and $\tilde{d}_0 > 0$, then $\tilde{d}_0 \delta < \tilde{d}_0 \leq 1+\alpha$, and consequently
\begin{align*}
\int_{|y| \leq L} |\Theta (y)|^p dy \leq C L^{D - \tilde{d}_0 (1+\delta)}.
\end{align*}
Hence, repeating this process $n$ times, for any given $n \in \N$, we arrive at
\begin{align}\label{eq_final_thm2_new}
\int_{|y| \leq L} |\Theta (y)|^p dy \leq C L^{D - \tilde{d}_0 (1+\delta + \ldots + \delta^n)}.
\end{align}
Since $1 + \delta + \ldots + \delta^n \rightarrow \frac{1}{1 -\delta} = r -p$ as $n \to \infty$ and $\gamma < \beta(r-p)$, we obtain that
\begin{align*}
D - \tilde{d}_0(r-p) = D - [(D -\gamma)(1-\delta) + \beta](r-p) = \gamma - \beta(r-p) < 0.
\end{align*}
Therefore, there exists $n$ sufficiently large such that the power of $L$ in \eqref{eq_final_thm2_new} is negative. This implies
that $\Theta \equiv 0$ in $\R^2$, which is a contradiction with our initial assumption that $\Theta \not\equiv 0$ in $\R^2$. This concludes the proof.

\subsection{Proof of \cref{thm:main4}}

We follow a similar proof as in \cref{thm:main3} and make the appropriate modifications. As such, to prove \ref{cor:2:i}, we first take $M$ to be a positive constant such that $|\Theta (y)| \lesssim |y|^{-\sigma}$ for all $|y| \geq M$. Since $\Theta \in \mathcal{C}^{\beta} (\R^2)$, we obtain that for all $L > 0$ and $r > \frac{2}{\sigma}$ 
\begin{align*}
\int_{|y| \leq L} |\Theta (y)|^r dy \leq \int_{|y| \leq M} |\Theta (y)|^r dy + \int_{|y| \geq M} \frac{1}{|y|^{r\sigma}} dy \leq C.
\end{align*}
Then, denoting $p_1:=  \max \left\{1, \frac{2}{\sigma}\right\}$, it follows that the assumptions of \cref{thm:main2} are satisfied with $p = p_1$, $r = p_1 + 1$ and $\gamma =0$. As a consequence, the values of $\alpha$ admitting nontrivial profiles $\Theta$ must belong to the interval $\left[\beta -1 + \frac{2}{p_1+1}, \beta -1 + \frac{2}{p_1}\right]$. Moreover, since the assumptions of \cref{thm:main2} are also verified with $p = p_1 + k$, $r= p_1+ k+1$ and $\gamma =0$, for any $k >0$, then such $\alpha$ must also belong to $\left[\beta -1 + \frac{2}{p_1+k + 1}, \beta -1 + \frac{2}{p_1 + k}\right]$ for any $k >0$. Taking $k \geq 2$, it follows that $\alpha \in \left[\beta -1 + \frac{2}{p_1 +k+1}, \beta -1 + \frac{2}{p_1 + k} \right] \cap \left[\beta -1 + \frac{2}{p_1 + 1}, \beta -1 + \frac{2}{p_1}\right] = \emptyset$, and we deduce that  $\Theta \equiv 0$ in $\R^2$, as desired.

Regarding item \ref{cor:2:ii}, let us now consider $M > 0$ such that $|\Theta (y)| \lesssim |y|^{\sigma}$ for all $|y| \geq M$, and fix an arbitrary $p \in [1, \infty)$ and $r \geq p+ 1$. Similarly as in \eqref{int:DTheta:Lr}, we have that for $L \gg 1$
\begin{align}\label{eq3teo2}
\int_{|y| \leq L} |\Theta (y)|^r dy \leq \int_{|y| \leq M} |\Theta (y)|^r dy + \int_{M \leq |y| \leq L} |y|^{r\sigma} dy \leq C L^{\sigma r +2},
\end{align}
where we again used the fact that $\Theta$ is a continuous function in $\R^2$ to bound the first integral.
Choosing $r$ sufficiently large such that $r > (2+ \beta p)/(\beta - \sigma)$, so that $\sigma r + 2 < \beta (r - p)$, it follows that the assumptions of \cref{thm:main2} are satisfied with any fixed $p \geq 1$, such choice of $r$, and $\gamma = \sigma r + 2$.	
Thus, the values of $\alpha$ admitting nontrivial profiles $\Theta$ in this case belong to the interval $[\beta - 1 - \sigma, \beta - 1 + \frac{2}{p}]$, and the corresponding nontrivial $\Theta$ satisfies \eqref{growth-profile-thm2}. But since $|\Theta(y)| \gtrsim 1$ for $|y| \gg 1$, then we may argue as in \eqref{eq6teo2} to deduce that in fact $\alpha \in [\beta - 1 - \sigma, \beta - 1]$, which concludes the proof.

\section{Auxiliary Results}\label{appendix}

In the following result, \cref{lemmamain}, we prove the upper bound for the function $\vt$ given in \eqref{vt} that was used at various steps in the proof of \cref{thm:main} regarding the case $\beta \in (1,2)$. Specifically, this lemma shows that the growth assumption \eqref{profile-hipothesis-derivada} on the $L^r$ norm of the gradient of the profile, $\nabla\Theta$, yields an upper bound on the $L^r$ norm of $\vt$ over a certain annulus in $\R^2$. The subsequent \cref{lemmamain2} shows an analogous result in the case $\beta \in (0,1]$ by relying instead on assumption \eqref{profile-hipothesis-function} on $\Theta$, which was used in the proof of \cref{thm:main2}.

Let us recall the definition of the set $\Ay$ given in \eqref{ball-defi}, namely
\begin{align}\label{ball_apendix}
    \Ay &:= \left\{ t \in [t_1,t_2] \,:\, \frac{\rho}{8} \frac{1}{|y|} \leq (T -t)^{\frac{1}{1+\alpha}} \leq \frac{\rho}{4} \frac{1}{|y|}\right\}
    \notag\\
    &= \left\{ t \in [t_1,t_2] \,:\,   T - \frac{(\rho/4)^{1 + \alpha}}{|y|^{1+ \alpha}} \leq t \leq T - \frac{(\rho/8)^{1 + \alpha}}{|y|^{1+ \alpha}} \right\},
\end{align}
for fixed $y \in \R^2 \setminus \{0\}$ and $0<t_1<t_2<T$.

\begin{lemma}\label{lemmamain}
	Let $\beta \in (1,2)$ and $\Theta \in \mathcal{C}^{1}(\R^2)$.  Suppose that for some $r\in [1,\infty)$ and $\gamma \in \mathbb R$, it holds
	\begin{equation}\label{secondlemmahypothesis}
		\int_{|y| \leq L} |\nabla \Theta|^r dy \lesssim L^\gamma \quad \textrm{for all } L \gg 1.
	\end{equation}
	Then, the function $\vt$ defined by
	\begin{align}\label{def:vt}
		\vt(y) = \int_{t_1}^{t_2} \bigg| \int_{\R^2} 
		\frac{1}{|y -z|^{\beta}} \nabla^\perp \Theta (z) \phi_\rho (z(T-t)^{\frac{1}{1+\alpha}}) dz \bigg| \ind_{\Ay}(t) dt,
	\end{align}
	with $0<t_1<t_2<T$ and $\Ay$ as given in \eqref{ball_apendix}, satisfies the following estimate
	\begin{equation}\label{lemma2}
		\int_{L \leq |y| \leq 2L} |\vt(y)|^r dy \lesssim L^{\gamma + r(1 -\alpha -\beta)} \quad \textrm{for all } L \gg 1.
	\end{equation}
\end{lemma}
\begin{proof}
Denote the kernel from \eqref{def:vt} by 
\begin{align*}
	\overline{K}_\beta(y) = \frac{1}{|y|^\beta}, \quad y \in \R^2 \backslash \{0\}.
\end{align*}
From the definitions of $\vt$ and $\Ay$, we have
\begin{align}\label{eq1lemma1}
	&\left( \int_{L \leq |y| \leq 2L} |\vt(y)|^r dy \right)^{\frac{1}{r}} \notag \\
	&\quad= \left( \int_{L \leq |y| \leq 2L}  \left(  \int_{t_1}^{t_2} \left| \int_{\R^2} 
	\overline{K}_\beta(y-z) \nabla^\perp \Theta (z) \phi_\rho (z(T-t)^{\frac{1}{1+\alpha}}) dz \right| \ind_{\Ay}(t) dt  \right)^r  dy \right)^{\frac{1}{r}} \notag \\
	&\quad\leq \left( \int_{L \leq |y| \leq 2L}  \left(  \int_{T - \frac{(\rho/4)^{1 + \alpha}}{|y|^{1+ \alpha}}}^{T - \frac{(\rho/8)^{1 + \alpha}}{|y|^{1+ \alpha}}} \left| \int_{\R^2} 
	\overline{K}_\beta(y-z) \nabla^\perp \Theta (z) \phi_\rho (z(T-t)^{\frac{1}{1+\alpha}}) dz \right|  dt  \right)^r  dy \right)^{\frac{1}{r}} \notag \\
	&\quad=  \left( \int_{L \leq |y| \leq 2L}  \left(  \int_{T - \frac{(\rho/4)^{1 + \alpha}}{|y|^{1+ \alpha}}}^{T - \frac{(\rho/8)^{1 + \alpha}}{|y|^{1+ \alpha}}} \left| \int_{\R^2} 
	\overline{K}_\beta(y-z) \ind_{B_{18L}(0)}(y-z) \nabla^\perp \Theta (z) \phi_\rho (z(T-t)^{\frac{1}{1+\alpha}}) dz \right|  dt  \right)^r  dy \right)^{\frac{1}{r}},
\end{align}
where as before $B_{18L}(0)$ denotes the ball of radius $18L$ centered at the origin, and we used that for $|y|\leq 2L$, $|z| \leq \rho(T-t)^{-\frac{1}{1+\alpha}}$ and $t\leq  T - (\frac{\rho/8}{|y|})^{1+\alpha} $, it holds that $|y-z| \leq 2L + \rho(T-t)^{-\frac{1}{1+\alpha}} \leq 2 L + 8 |y| \leq 18L$.

Applying Minkowski and Young's convolution inequality, we obtain
	\begin{align}\label{eq:last:lemmamain2}
	&\bigg( \int_{L \leq |y| \leq 2L} |\vt(y)|^r dy \bigg)^{\frac{1}{r}} \notag \\
	&\leq  \int_{ T - (\frac{\rho/8}{L})^{1+\alpha}}^{{ T - (\frac{\rho/16}{L})^{1+\alpha} }} \bigg(\int_{L \leq |y| \leq 2L} \bigg|\int_{\mathbb R^2} \overline{K}_\beta(y-z)\ind_{B_{18L}(0)}(y-z) \nabla^\perp \Theta(z) \phi_{\rho} ((T-t)^{\frac{1}{1+\alpha}} z)  dz \bigg|^r  dy\bigg)^{\frac{1}{r}} dt \notag \\
	&\leq \int_{ T - (\frac{\rho/4}{L})^{1+\alpha}}^{{ T - (\frac{\rho/16}{L})^{1+\alpha} }} \bigg(\int_{\R^2} |[ (\overline{K}_\beta \ind_{B_{18L}(0)}) \ast ((\nabla^\perp \Theta \phi_{\rho} ((\cdot)(T-t)^{\frac{1}{1+\alpha}} ))](y) |^r dy\bigg)^{\frac{1}{r}} dt \notag \\
	&\leq \int_{ T - (\frac{\rho/4}{L})^{1+\alpha}}^{{ T - (\frac{\rho/16}{L})^{1+\alpha} }} ||\overline{K}_\beta  \ind_{B_{18L}(0)} ||_{L^1(\R^2)} ||\nabla^\perp \Theta  \phi_{\rho} ((\cdot)(T-t)^{\frac{1}{1+\alpha}})||_{L^r(\R^2)} dt  
	\notag \\
	&\leq \int_{ T - (\frac{\rho/4}{L})^{1+\alpha}}^{{ T - (\frac{\rho/16}{L})^{1+\alpha} }} \bigg( \int_{|y| \leq 18L} \frac{1}{|y|^\beta} dy \bigg) \bigg( \int_{|y| \leq \rho (T-t)^{-\frac{1}{1+\alpha}}} |\nabla^\perp \Theta (y)|^rdy \bigg)^{\frac{1}{r}} dt \notag \\
	&\leq \bigg( \int_{|y| \leq 18L} \frac{1}{|y|^\beta} dy \bigg) \int_{ T - (\frac{\rho/4}{L})^{1+\alpha}}^{{ T - (\frac{\rho/16}{L})^{1+\alpha} }}  \bigg( \int_{|y| \leq 16 L} |\nabla\Theta (y)|^rdy \bigg)^{\frac{1}{r}} dt. 
\end{align}

Since $\beta \in (1,2)$, then $\int_{|y| \lesssim L} |y|^{-\beta} dy \lesssim L^{2 - \beta}$. Moreover, invoking assumption \eqref{secondlemmahypothesis}, it follows that for all $L$ sufficiently large
\begin{align*}
	\bigg( \int_{L \leq |y| \leq 2L} |\vt(y)|^r dy \bigg)^{\frac{1}{r}} 
	\lesssim  L^{2-\beta} \int_{ T - (\frac{\rho/4}{L})^{1+\alpha}}^{{ T - (\frac{\rho/16}{L})^{1+\alpha} }} L^{\frac{\gamma}{r}} dt
	\lesssim L^{\frac{\gamma}{r} + 1-\alpha - \beta},
\end{align*}
which proves \eqref{lemma2}.
\end{proof}

\begin{lemma}\label{lemmamain2}
Let $\beta \in (0,1]$ and $\Theta \in \mathcal{C}^{\beta}(\R^2)$.  Suppose that for some $r\in [1,\infty)$ and $\gamma \in \mathbb R$, it holds
\begin{equation}\label{firstlemmahypothesis}
    \int_{|y| \leq L} |\Theta(y)|^r dy \lesssim L^\gamma, \quad \textrm{for all } L \gg 1.
\end{equation}
Then, the function $\tU$ defined by
\begin{align}\label{eqUlemma1}
    \tU(y) =  \int_{t_1}^{t_2} P.V. \int_{\R^2} \left| \frac{(y-z)^\perp}{|y-z|^{2 +\beta}} \Theta(z) \phi_{\rho}(z(T-t)^{\frac{1}{1+\alpha}})  dz \right| \ind_{\Ay}(t) dt 
\end{align}
where $0<t_1<t_2<T$, satisfies the following estimate
\begin{equation}\label{lemma2-betamenor1}
\int_{L \leq |y| \leq 2L} |\tU(y)|^r dy \lesssim L^{\gamma - r(\alpha +\beta)}, \quad \textrm{for all }  L \gg 1.
\end{equation}

\end{lemma}

\begin{proof}
Let us first assume that $\beta \in (0,1)$. Note that, in this case, the kernel $K_\beta(x) = y^\perp |y|^{-(2+\beta)}$, $y \in \R^2 \setminus\{0\}$, from \eqref{eqUlemma1} is integrable near the origin. Namely, 
\begin{align}\label{int:kernel:1}
	\int_{|y|\lesssim L} |K_\beta(y)| dy \leq \int_{|y|\lesssim L} \frac{1}{|y|^{1 + \beta}} dy \lesssim L^{1 - \beta}.
\end{align}
We may thus apply similar arguments as in the proof of \cref{lemmamain} to arrive at 
\begin{align*}
	 \bigg( \int_{L \leq |y| \leq 2L}  |\tU (y)|^r dy \bigg)^{\frac{1}{r}} 
	 &\leq \bigg( \int_{|y| \leq 18 L} \frac{1}{|y|^{1+\beta}} dy \bigg) \int_{T - \left(\frac{\rho/4}{L}\right)^{1+\alpha}}^{T -\left(\frac{\rho/16}{L}\right)^{1+\alpha}}  \bigg(\int_{|y| \leq 16L } |\Theta (z) |^r dz \bigg)^\frac{1}{r} dt.
\end{align*}
Thus, it follows from \eqref{int:kernel:1} and assumption \eqref{firstlemmahypothesis} that
\begin{align*}
	 \bigg( \int_{L \leq |y| \leq 2L}  |\tU (y)|^r dy \bigg)^{\frac{1}{r}} 
	 \lesssim L^{1 - \beta} \int_{T - \left(\frac{\rho/4}{L}\right)^{1+\alpha}}^{T -\left(\frac{\rho/16}{L}\right)^{1+\alpha}} L^{\frac{\gamma}{r}} dt
	 \lesssim L^{\frac{\gamma}{r} - \alpha - \beta},
\end{align*}
as desired. 

The proof for the case $\beta =1$ was done in \cite[Lemma 2.2]{Xue16}. Here we note that \eqref{int:kernel:1} no longer holds. However, in this case the kernel $K_\beta$ is a Calder\'on-Zygmund operator, and hence $\| K_\beta \ast f \|_{L^q(\R^2)} \lesssim \|f\|_{L^q(\R^2)}$ for any $f \in L^q(\R^2)$ and $1 < q < \infty$. Replacing the use of Young's convolution inequality in \eqref{eq:last:lemmamain2} with this property, one may then proceed with analogous arguments and conclude \eqref{lemma2-betamenor1}. 
\end{proof}

\section*{Acknowledgements}
ACB received support under the grants \#2019/02512-5, S\~ao Paulo Research Foundation (FAPESP), and FAEPEX/UNICAMP. RMG received support under the grants \#2018/22385-5 and \#2022/00948-3, S\~ao Paulo Research Foundation (FAPESP). CFM was supported by the grants  NSF DMS-2009859 and NSF DMS-2239325. The authors would also like to thank Mimi Dai for helpful comments on an earlier version of this manuscript.

\bibliographystyle{plain}
\bibliography{refs}

\newpage

\setlength{\columnsep}{.5in}
\begin{multicols}{2}
\noindent
Anne Bronzi\\ {\footnotesize
Instituto de Matem\'atica, Estat\'istica e Computa\c c\~ao Cient\'ifica\\
Universidade Estadual de Campinas (UNICAMP), Brazil\\
Email: \href{mailto:acbronzi@unicamp.br}{\nolinkurl{acbronzi@unicamp.br}}} \\[.2cm]

\noindent Ricardo Guimar\~aes\\
{\footnotesize
Instituto de Matem\'atica, Estat\'istica e Computa\c c\~ao Cient\'ifica\\
Universidade Estadual de Campinas (UNICAMP), Brazil\\
Email: \href{mailto:r192301@dac.unicamp.br}{\nolinkurl{r192301@dac.unicamp.br}}} \\[.2cm]

\columnbreak
\noindent Cecilia F. Mondaini \\
{\footnotesize
  Department of Mathematics\\
  Drexel University, USA\\
Email: \href{mailto:cf823@drexel.edu}{\nolinkurl{cf823@drexel.edu}}} \\[.2cm]
\end{multicols}

\end{document}